\documentclass{amsart}
\usepackage[left=3cm,right=3cm]{geometry}
\usepackage[T1]{fontenc}
\usepackage[utf8]{inputenc}

\usepackage{amsmath}	% Mathe: Für align-Umgebung
\numberwithin{equation}{section}

\usepackage{amsfonts}   % Mathe für mathbb
\usepackage[alphabetic]{amsrefs}
\usepackage{tikz-cd}  % Kommutierende Diagramme
\usepackage{amssymb}

\usepackage{amsthm}
\usepackage{mathtools}
\usepackage{enumerate}   %Verschiedene Aufzählungen

\usepackage[colorlinks]{hyperref}

\hypersetup{
	colorlinks,
	citecolor=black,
	filecolor=black,
	linkcolor=black,
	urlcolor=black
}

\newtheorem{lemma}{Lemma}[section]
\newtheorem{cor}{Corollary}
\newtheorem{rem}{Remark}[section]
\newtheorem{thm}{Theorem}

\newtheorem{definition}{Definition}

\newcommand{\IN}{\mathbb{N}}
\newcommand{\IZ}{\mathbb{Z}}
\newcommand{\IQ}{\mathbb{Q}}
\newcommand{\IR}{\mathbb{R}}
\newcommand{\IC}{\mathbb{C}}

\newcommand{\fa}{\mathfrak{a}}

\newcommand{\CB}{\mathcal{B}} 
\newcommand{\CN}{\mathcal{N}}

\newcommand{\CH}{\mathcal{H}}

\newcommand{\spec}{\mathrm{spec}}

\newcommand{\SL}{\mathrm{SL}}             % Fraktur sl
 
\newcommand{\GL}{\mathrm{GL}}
\newcommand{\PGL}{\mathrm{PGL}}
 
\renewcommand{\mod}[1]{\ \textnormal{mod } #1}

\DefineSimpleKey{bib}{primaryclass}{}
\DefineSimpleKey{bib}{archiveprefix}{}

\BibSpec{arXiv}{%
  +{}{\PrintAuthors}{author}
  +{,}{ \textit}{title}
  +{}{ \parenthesize}{date}
  +{,}{ arXiv }{eprint}
  +{,}{ primary class }{primaryclass}
}

\title{Symmetric square $L$-functions on $\mathrm{GL}_3$}
\author{Johannes Linn }
\date{}
\thanks{The author thanks Prof.\ Valentin Blomer for helpful discussions.}
\address{Max Planck Institute for Mathematics Bonn, Vivatsgasse 7, 53111 Bonn, Germany}
\email{linn@mpim-bonn.mpg.de}

\begin{document}
\begin{abstract}
    We give an asymptotic formula with a power-saving error term for the twisted first moment of symmetric square $L$-functions on $\GL_3$ in the spectral aspect. We apply this to obtain non-vanishing results and lower bounds of the expected order of magnitude for even moments, supporting the random matrix model for a unitary ensemble of $L$-functions. The main ingredients are the $\GL_3$ Kuznetsov formula, an asymmetric approximate functional equation, and strong bounds for the integral transforms appearing in the Kuznetsov formula.
\end{abstract}
\maketitle
\section{Introduction}
\subsection{Twisted moments in the spectral aspect.} 
Moments of central values of automorphic $L$-functions are an important object in analytic number theory. They provide a quantitative way of studying non-vanishing and the distribution of $L$-values, and the size of $L$-functions on the critical line. While the theory is highly developed for families on $\GL_1$ and $\GL_2$, considerably less is known in higher rank. In particular, genuinely varying $\GL_3$-families present substantial new difficulties, both on the arithmetic side, through higher-rank Kloosterman sums, and on the analytic side, through the intricate archimedean integral transforms appearing in the $\GL_3$ Kuznetsov formula. 

In this paper, we study such a family of degree-six $L$-functions in the spectral aspect. More precisely, we consider central values of symmetric square $L$-functions attached to level-one cuspidal automorphic representations of $\PGL_3$. Let $\CB$ denote the set of these representations. 

A preliminary issue is that the symmetric square $L$-function is not known to be automorphic in this generality. Nevertheless, its meromorphic continuation and functional equation follow from its relation to Rankin–Selberg $L$-functions
\[
    \Lambda(s,\pi,\mathrm{sym}^2)\Lambda(s,\widetilde{\pi})
    =
    \Lambda(s,\pi\otimes \pi)\ .
\]
Using recent work of Gan \cite{Ga25} on the adjoint lift from $\GL_3$ to $\GL_8$, we show that the completed symmetric square $L$-function is entire except for the possible poles at $s=0$ and $s=1$ that occur for symmetric square lifts; see Lemma \ref{lemma:lfuncisentire}.

Our main result is an asymptotic formula, with a power-saving error term, for the first moment of the central values of these $L$-functions twisted by Fourier-Whittaker coefficients $A_\pi(m_1,m_2)$ for a spectral family that consists essentially of all forms $\pi\in\CB$ with Langlands parameter $\mu_\pi$ of size $T$ inside a small ball with radius $T^\delta$ for some small positive constant $\delta$. We require the center of the ball $\tilde{\mu}$ to be in generic position, meaning it is away from the walls of the Weyl chamber and away from self-dual forms; see (\ref{eq:genericposition}).
We require $\tilde{\mu}$ to be away from the walls, as the spectral measure drops near the walls and spectral averages become less powerful. We require $\tilde{\mu}$ to be away from self-dual forms because they have lower analytic conductor, which leads to a modified approximate functional equation that would need different treatment; see Remark \ref{rem:selfdualforms}.
\begin{thm}
\label{thm:MainTheorem}
There exists $\delta_0>0$ such that for every $0<\delta<\delta_0$ there exist constants $A_0,\eta,\kappa>0$ with the following property. Let $\tilde{\mu}$ be in generic position, as in (\ref{eq:genericposition})
below, and let $h\in \tilde{\CH}(A_0,\tilde{\mu},\delta)$ be a test function, as defined in section \ref{subsec:kuznetsov_testfunction}, and let $m_1,m_2\in \IZ^+$ with $m_1,m_2\leq \|\tilde{\mu}\|^\eta$. Then the following holds:
\[
    \sum_{\pi\in \mathcal{B}}\overline{A_\pi(m_1,m_2)}L(1/2,\pi,\mathrm{sym}^2)\frac{h(\mu_\pi)}{\CN(\pi)}=\zeta(\frac{3}{2})\frac{1}{192\pi^5}\delta_{\substack{m_1=\square\\m_2=\square}}\frac{\mathrm{vol}_\spec(h)}{(m_1^2m_2)^{1/4}}+O(\|\tilde{\mu}\|^{3-\kappa})\ ,
\]
where
\[
    \mathrm{vol}_\spec(h):=\int_{\Re\mu=0}h(\mu)\spec(\mu)\mathrm{d}\mu
\]
is the spectral volume of the family, with $\spec(\mu)$ being the spectral measure defined in (\ref{eq:def_spec}) and $\CN(\pi)$ being normalizing factors defined in Section \ref{sec:normfactors}.
\end{thm}

The test function $h$ can be viewed as a holomorphic approximation to the indicator function of the ball mentioned above around $\tilde{\mu}$ with a radius of $\|\tilde{\mu}\|^\delta$.

A predecessor to this result is the work of Blomer and Comtat \cite{BC25}, who obtained an asymptotic formula with a power-saving error term for a twisted first moment of symmetric square $L$-functions on $\GL_3$ in the prime-level aspect. Their family consists of automorphic representations of prime level $p$, and the main term is of size comparable to the volume of the level aspect. In contrast, we keep the level fixed, equal to one, and vary the spectral parameter $\tilde{\mu}$.

Working in the spectral aspect instead of the level aspect changes the main focus points significantly. In the level aspect, a substantial part of the difficulty comes from the ramified finite places: local computations at the prime level, root numbers, the structure of oldforms, and the level-$p$ Kuznetsov formula. In the level-one spectral problem, these finite-place complications mostly disappear and only remain in the treatment of the Kloosterman sums, but in addition, we have to solve a delicate archimedean problem. The test function must simultaneously localize the desired spectral family, interact suitably with the gamma factors in the approximate functional equation for $L(1/2,\pi,\mathrm{sym}^2)$, and allow sufficient control of the integral transforms occurring in the $\GL_3$ Kuznetsov formula.

The main ideas in our work are therefore more closely related to the spectral-aspect analysis of Blomer and Buttcane \cite{BB20}, who proved a subconvexity bound for the central value of the standard $L$-function $L(1/2,\pi)$ on $\GL_3$ at level one. As in their work, a genuine $\GL_3$ spectral problem requires one to exploit various forms of cancellation on the geometric side of the Kuznetsov formula. In particular, the long Weyl element and the associated integral kernel require a fine analysis of oscillatory integral transforms.

The $\GL_3$ Kuznetsov formula itself has a long history and is presented in various forms in the literature, including Bump, Friedberg, and Goldfeld \cite{BFG88}, Li \cite{Li10}, Blomer \cite{Bl13}, Blomer-Buttcane-Maga \cite{BBM17}, and Buttcane \cite{Bu16}. In this paper, we are using the version of Li \cite{Li10} as presented in \cite{Bu16} and \cite{BB20}.

Before going over the main ideas of the proof, we present a few applications of Theorem \ref{thm:MainTheorem}.
As a direct application, we obtain a non-vanishing result:
\begin{cor}
\label{cor:nonvanishing}
    There exists $\delta_0>0$ such that for every $0<\delta<\delta_0$ there exist constants $c,c^\prime>0$ such that for any $\pi\in\mathcal{B}$ with $\|\mu_\pi\|\geq c$ and $\mu_\pi$ in generic position, there exist at least $c^\prime \|\mu_\pi\|^{\frac{3}{2}}$ forms $\pi^\prime$ with $\|\mu_{\pi^\prime}-\mu_\pi\|\leq \|\mu_\pi\|^{\delta}$ and $L(1/2,\pi^\prime,\mathrm{sym}^2)\neq 0$.
\end{cor}
This follows from choosing $m_1=m_2=1$ and a test function $h$ that is positive on the tempered spectrum, as in (\ref{eq:typical_testfunction}), together with the convexity bound of
\[
    L(1/2,\pi,\mathrm{sym}^2)\ll\|\mu_\pi\|^{\frac{3}{2}}
\]
and because $\frac{1}{\CN(\pi)}\ll_\varepsilon \|\mu_\pi\|^\varepsilon$; see Lemma \ref{lemma:normfactor_bounds}. This also implies that of the $\Theta(T^5)$ forms with spectral parameter of size $T$ in generic position, at least $T^{\frac{7}{2}}$ satisfy $L(1/2,\pi,\mathrm{sym}^2)\neq 0$.

A more involved application needing the twist by $\overline{A_\pi(m_1,m_2)}$ is the lower bound for all even moments of the expected order of magnitude. This uses a method of Rudnick and Soundararajan \cite{RS05} as presented in \cite{BC25}.
\begin{cor}
\label{cor:lowerbound_moments}
There exists $\delta_0>0$ such that for every $0<\delta<\delta_0$ the following holds: Let $h\in \tilde{\CH}(A_0,\tilde{\mu},\delta)$ be non-negative on the tempered spectrum and $\tilde{\mu}$ in generic position. Then, for any fixed $k\in\IZ^+$, we have
\[
    \Big|\sum_{\pi\in\CB}|L(1/2,\pi,\mathrm{sym}^2)|^{2k}\frac{h(\mu_\pi)}{\CN(\pi)}\Big|\gg_k \mathrm{vol}_\spec(h) (\log \|\tilde{\mu}\|)^{k^2}\ .
\]
\end{cor}
The exponent is in agreement with the random matrix model for a unitary ensemble of $L$-functions. The proof is analogous to \cite{BC25}*{Proof of Corollary 3}, with the only two differences being that we include the even Fourier coefficients into the sums from the start and that the size of our family, and thus the diagonal term of the Kuznetsov formula and the main term of Theorem \ref{thm:MainTheorem}, are $\mathrm{vol}_\spec(h)$ and not $p^2$.

The left-hand side of Corollary \ref{cor:lowerbound_moments} has the added absolute values, as the non-tempered spectrum could have a non-real contribution. If this is the case, the contribution is negligible and does not affect the bound.

\subsection{Main ideas and sketch}
Let $\pi\in\CB$ with spectral parameter $\tilde{\mu}$ in generic position and let $T:=\|\tilde{\mu}\|$. The test function $h$ will be a holomorphic function that mimics a smooth, compactly supported function with support in a ball of radius $T^\delta$ around $\tilde{\mu}$. The following is just a sketch of the main arguments, and the details are often much more involved. For the sketch, we let $m_1=m_2=1$.

We want to compute
\[
    \sum_{\pi\in \mathcal{B}} \overline{A_\pi(1,1)}L(1/2,\pi,\mathrm{sym}^2)\frac{h(\mu_\pi)}{\CN(\pi)} \ .
\]
The first step is to use an approximate functional equation, Theorem \ref{thm:approxfunctionalequation}, to access the central value. This leads to
\[
    \sum_{\pi\in \mathcal{B}}T^{-\frac{3}{2}}\sum_{\substack{b\asymp1\\c\asymp T^3}} A_{\pi}(b^2,c^2)\overline{A_\pi(1,1)}\frac{h(\mu_\pi)}{\CN(\pi)}\ ,
\]
where here and throughout this sketch we do not display smooth weight functions. To apply the Kuznetsov formula, we have to add the rest of the spectrum, which we need to bound as we do not have positivity. But because the continuous part of the spectrum is spectrally sparser, its contribution is absorbed into the error term.

After applying the Kuznetsov formula, the diagonal gives us the main term, and we need to bound the contribution of the other Weyl elements. The most difficult contribution comes from the long Weyl element. Its contribution is
\[
    T^{-\frac{3}{2}}\sum_{\substack{b\asymp1\\c\asymp T^3}}\sum_{D_1,D_2}\frac{S(c^2,1,1,b^2;D_1,D_2)}{D_1D_2}\Phi_{w_6}\left(\frac{c^2D_2}{D_1^2},\frac{b^2D_1}{D_2^2}\right)\ ,
\]
where $S(c^2,1,1,b^2;D_1,D_2)$ is a certain Kloosterman sum and $\Phi_{w_6}$ is the integral transform of $h$ defined by an integral kernel $K_{w_6}$ that is essentially a double Mellin inverse of a product of Gamma factors and trigonometric functions; see Section \ref{sec:inttransforms}.

Heuristically, the size of $\Phi_{w_6}$ is $T^{\frac{3}{2}}$, but for small values of $y_1,y_2$, the size of $\Phi_{w_6}(y_1,y_2)$ is much smaller, as shown in Section \ref{section:LemmaWeightFunctions}. This enables us to truncate the $D_1,D_2$ sums to $D_2\asymp 1$ and $D_1\asymp T^3$. 

In the remaining cases, we first sum over $b,c$ in dyadic ranges and then perform Poisson summation in $c$ modulo $D_1$. This shrinks the effective range of $c$ down to $T^\varepsilon$. We now have to estimate certain Fourier transforms of the Kloosterman sums and the integral transform. We can prove essentially square root cancellation in the sum over Fourier transforms of the Kloosterman sum because the Fourier transforms barely oscillate but serve as an indicator for certain congruences between the summation variables; see Lemma \ref{lemma:exp_sum_bound}. Using stationary phase methods, we achieve a large enough power saving in the integral transform such that summing over all remaining variables leaves us with something bounded by $T^{3-\kappa}$.

The two main hidden difficulties that appear in practice are:

Firstly, the approximate functional equation contains two parts, one containing an oscillating factor coming from $L_\infty(s,\pi,\mathrm{sym}^2)$. To avoid dealing with the added oscillation in the critical ranges, we introduce a small constant $\alpha$ to obtain an asymmetric approximate functional equation, shortening the sum in the second part. We also have to add a smooth function with specified zeros that cancel the poles of $L_\infty(s,\pi,\mathrm{sym}^2)$ and enable us to shift contours during the proof. 

Secondly, a straightforward stationary phase argument for the integral transform of the long Weyl element is not enough, and one needs to perform a more delicate analysis and use a two-dimensional stationary phase argument to obtain the needed bounds; see Lemma \ref{lemma:bound_fourier_long}.

In addition to these main difficulties, a lot of care is needed to treat non-generic cases, as the above arguments often fail for technical reasons or because some variables are of the wrong size.

\subsection{Setup and notation}
\label{sec:Sym_setup}

We rely on standard background on automorphic forms and Hecke theory on $\GL_n$, as in \cite{Go06}. We use this section to recall the most important results and fix some notations and conventions.

\subsubsection{Satake and spectral parameters}
Let $\pi\in\CB$. We normalize the Langlands parameter $\mu_\pi$ such that  the Ramanujan conjecture predicts $\Re\mu_\pi=0$. For each finite place, we have the Satake parameters $\alpha_p,\beta_p,\gamma_p$ normalized such that the Ramanujan conjecture predicts $|\alpha_p|=|\beta_p|=|\gamma_p|=1$. They satisfy
\begin{equation}
    \label{eq:triv_central_character}
    \mu_1+\mu_2+\mu_3=0,\quad \alpha_p\beta_p\gamma_p=1 \quad(\textnormal{trivial central character})
\end{equation}
and
\begin{equation}
    \label{eq:unitarity}
    \{-\mu_1,-\mu_2,-\mu_3\}=\{\overline{\mu_1},\overline{\mu_2},\overline{\mu_3}\},\quad \{\alpha_p^{-1},\beta_p^{-1},\gamma_p^{-1}\}=\{\overline{\alpha_p},\overline{\beta_p},\overline{\gamma_p}\}\quad (\textnormal{unitarity})
\end{equation}
and
\begin{equation}
    \label{eq:boundstowardsram}
    |\Re\mu_1|,|\Re\mu_2|,|\Re\mu_3|\leq \theta_3,\quad |\alpha_p|,|\beta_p|,|\gamma_p|\leq p^{\theta_3}\quad (\textnormal{bounds towards Ramanujan})
\end{equation}
with $\theta_3=\frac{1}{4}-\frac{1}{130}$ due to Gan \cite{Ga25}; see Lemma \ref{lemma:lfuncisentire} below for more details.

We parametrize the space of possible spectral parameters as follows:

For $0\leq c\leq \infty$ define
\[
    \Lambda_c:=\{\mu\in\IC^3:\mu_1+\mu_2+\mu_3=0,|\Re\mu_i|\leq c\}\ .
\]
We have $\mu_\pi\in\Lambda_{\theta_3}$, where $\theta_3$ is the best known bound towards Ramanujan. In the Lie algebra $\fa_\IC^*=\Lambda_\infty$, we will use two sets of coordinates, one being $\mu$ and the other being $\nu$ defined by
\begin{equation}
\label{eq:def_nu}
    \nu_1=\frac{1}{3}(\mu_1-\mu_2),\quad \nu_2=\frac{1}{3}(\mu_2-\mu_3),\quad \nu_3=\frac{1}{3}(\mu_3-\mu_1)\ .
\end{equation}
We define the Weyl $\mathcal{W}$ group as
\[
    \mathcal{W}=\{I,w_2=\begin{psmallmatrix}1&&\\&&1\\&1&\end{psmallmatrix},w_3=\begin{psmallmatrix}&1&\\1&&\\&&1\end{psmallmatrix},w_4=\begin{psmallmatrix}&1&\\&&1\\1&&\end{psmallmatrix},w_5=\begin{psmallmatrix}&&1\\1&&\\&1&\end{psmallmatrix},w_6=\begin{psmallmatrix}&&1\\&1&\\1&&\end{psmallmatrix}\}\cong S_3\ .
\]
It acts on the $\mu_i$ by permutations.

We fix two constants $0<c_1<c_2$ for the rest of the paper, and any other constants may depend on them. We say $\mu\in \Lambda_\infty$ is in generic position if
\begin{equation}
    \label{eq:genericposition}
    c_1\leq \frac{|\mu_1|}{\|\mu\|},\frac{|\mu_2|}{\|\mu\|},\frac{|\mu_3|}{\|\mu\|}\leq c_2,\quad \textnormal{and}\quad c_1\leq \frac{|3\nu_1|}{\|\mu\|},\frac{|3\nu_2|}{\|\mu\|},\frac{|3\nu_3|}{\|\mu\|}\leq c_2\ .
\end{equation}
This describes two cones in each Weyl chamber away from the walls and self-dual forms and covers 99\% of all Maa\ss forms by Weyl's law \cite{LM09}*{Theorem 1.1} and \cite{BB20}*{Remark (1) p.\ 1442}. 

We fix $\tilde{\mu}$ in generic position with $T:=\|\tilde{\mu}\|$. Throughout the paper, our goal is to understand the behavior of the central value as $T$ grows, and everything should always be viewed for $T$ large enough.

We write $A\asymp B$ if there are constants $c,c^\prime$ such that $cA\leq B\leq c^\prime A$. Thus, we have
\[
    |\mu_i|\asymp|\nu_j|\asymp T
\]
for $1\leq i,j\leq 3$.

We also write $A\asymp_\varepsilon B$ if there are constants $c,c^\prime$ such that $cAT^{-\varepsilon}\leq B\leq c^\prime AT^\varepsilon$. This will often come up because two quantities do not have to be of the same size but cannot differ by a fixed power of $T$. We will call something negligible if it is in $O_B(T^{-B})$ for $B$ large enough such that the term containing the negligible factor can always be absorbed into the error term.

\subsubsection{The symmetric square L-function}
$A_\pi(m_1,m_2)$ are the Fourier-Whittaker coefficients of $\pi$, normalized such that $A_\pi(1,1)=1$. They are multiplicative in both entries and satisfy $A_\pi(m,n)=\overline{A_\pi(n,m)}$, as $\pi$ is everywhere unramified. They also satisfy $A_\pi(m,n)\ll_\varepsilon (mn)^{\theta_3+\varepsilon}$. For more details about the Hecke algebra for $\GL_3$, also in the level-aspect, we refer to \cite{BC25}*{1.2. The Hecke algebra}.

By \cite{BC25}*{Proposition 2 and 3}, we know that
\[
    L(s,\pi,\mathrm{sym}^2)=\zeta(3s)\sum_{b,c}\frac{A_\pi(b^2,c^2)}{(b^2c)^s}
\]
and
\[
    L_\infty(s,\pi,\mathrm{sym}^2)=\prod_{1\leq i\leq j\leq 3}\Gamma_\IR(s-\mu_i-\mu_j),\quad \Gamma_\IR(s)=\Gamma(\frac{s}{2})\pi^{-s/2}\ .
\]  
We also know that the arithmetic conductor and the root number are $1$, and therefore, the functional equation reads
\begin{equation}
    \label{eq:funcequation}
    \Lambda(s,\pi,\mathrm{sym}^2):=L(s,\pi,\mathrm{sym}^2)L_\infty(s,\pi,\mathrm{sym}^2)=\Lambda(1-s,\tilde{\pi},\mathrm{sym}^2),
\end{equation}
where $\tilde{\pi}$ is the contragredient representation.

As mentioned above, this $L$-function is only conjecturally automorphic. To work with it unconditionally, we need the following Lemma.

\begin{lemma}
\label{lemma:lfuncisentire}
Let $\pi$ be a level one cuspidal automorphic representation of $\PGL_3$ over $\IQ$. Then
\[
    |\alpha_p|,|\beta_p|,|\gamma_p|\leq p^{\theta_3},\quad |\Re\mu_1|,|\Re\mu_2|,|\Re\mu_3|\leq \theta_3
\]
with $\theta_3=\frac{1}{4}-\frac{1}{130}$.

And the completed $L$-function $\Lambda(s,\pi,\mathrm{sym}^2)$ is entire except for possible simple poles at $s=0$ and $s=1$ that occur if and only if $\pi$ is the symmetric square of a $\GL_2$ representation.
\end{lemma}
\begin{proof}
   \cite{Ga25}*{Theorem 12.1} states that $\theta_3=\frac{1}{4}-\frac{1}{130}$.

    We have from \cite{Ba97}*{Theorem 7} that $\Lambda^{(S)}(s,\pi,\mathrm{sym}^2)$ is entire except for possible poles at $s=1$ and $s=0$, where $S$ consists of the places $2$ and $\infty$. The above bound on $\theta_3$ implies that the local factors at $p=2$ and $\infty$  of the symmetric square $L$-function can only have poles in the region $\Re s<\frac{1}{2}-\frac{1}{65}$. 
    Thus, except for $s=0$ and $s=1$, the completed $L$-function $\Lambda(s,\pi,\mathrm{sym}^2)$ can only have poles in the region $\Re s<\frac{1}{2}-\frac{1}{65}$. Using the functional equation (\ref{eq:funcequation}), the existence of such a pole at $s$ would imply that $\Lambda(s,\tilde{\pi},\mathrm{sym}^2)$ has a pole at $1-s$, but $\Re(1-s)>\frac{1}{2}$, which contradicts our above arguments that also hold for the contragredient representation.
    
      The claim that the poles at $s=0$ and $s=1$ are simple and appear if and only if $\pi$ is a symmetric square lift from $\GL_2$ is proven in \cite{PPS89}*{p.\ 552}.
\end{proof}
\begin{rem}
    If we instead work with a form of a higher level or not over $\IQ$, we also have to account for the ramified places and treat their $L$-factors separately. An example of how this can be achieved can be found in \cite{BC25}*{Lemma 1 and Lemma 4}.
\end{rem}

\section{The Kuznetsov formula}
We quote the Kuznetsov formula by Li \cite{Li10} as presented in \cite{Bu16}*{Theorem 4} and \cite{BB20}*{Section 3}. 
\subsection{Normalizing factors}
\label{sec:normfactors}
We define the normalizing factors as in \cite{BB20}*{3.1 Normalizing factors}. To be precise, we use the normalizing factors $\CN(\pi)$ defined as the square of the ratio between Hecke eigenvalues and Fourier coefficients of $L^2$-normalized automorphic forms appearing in the spectral decomposition. For $\pi\in\mathcal{B}$, this is
\[
    \CN(\pi):=\|\phi\|^2\prod_{i=1}^3\cos\Big(\frac{3}{2}\pi \nu_{\pi,j}\Big)\,
\]
where $\phi$ is the arithmetically normalized Maa{\ss} form generating $\pi$, and $\nu$ is as in (\ref{eq:def_nu}).

We have 
\[
    \CN(\pi)\asymp \operatorname{res}_{s=1}L(s,\pi\times\tilde{\pi}) = L(1,\pi,\mathrm{Ad})\ ,
\]
see \cite{BB20}*{3.1. Normalizing factors}.

For non-cuspidal $\pi$, the proper analog of $\CN(\pi)$ is given by 
\[
    \frac{1}{16}\prod_{j=1}^3|\zeta(1+3\nu_j)|^2
\]
if $\pi$ is generated by a maximal Eisenstein series, and
\[
    8L(1,\mathrm{Ad}u)|L(1+3s,u)|^2
\]
if $\pi$ is generated by a maximal Eisenstein series $E(z,\frac{1}{2}+s,u)$ associated with an $\SL_2(\IZ)$ cusp form $u$; see \cite{Bu84}*{Chapter 8} and \cite{BB20}*{3.1. Normalizing Factors}.

\begin{lemma}
\label{lemma:normfactor_bounds}
For $\pi\in\CB$ we have
\[
    \|\mu_\pi\|^{-\varepsilon}\ll_\varepsilon \CN(\pi)\ll_\varepsilon\|\mu_\pi\|^{\varepsilon}\ .
\]
\end{lemma}
\begin{proof}
    The upper bound is immediate from the relation to $L(1,\pi,\mathrm{Ad})$ and is found in \cite{BB20}*{3.1. Normalizing factors}. The lower bound is quite delicate, and without more work, one would only expect $\|\mu_\pi\|^{-B}$ for some large $B$. But due to \cite{Ga25}, the adjoint lift to $\GL_8$ is automorphic, and thus we can use the arguments from \cite{Sa04}, specifically \cite{Sa04}*{(4)}, to prove a standard zero-free region for $L(s,\pi,\mathrm{Ad})$, implying a logarithmic lower bound for the value $L(1,\pi,\mathrm{Ad})$.
\end{proof}

\subsection{Choice of test function}
\label{subsec:kuznetsov_testfunction}
Our test functions are defined on $\Lambda_\infty$. With the spectral measure $\spec(\mu)$ defined by
\begin{equation}
\label{eq:def_spec}
    \spec(\mu)\mathrm{d}\mu:=\prod_j\Big(3\nu_j\tan(\frac{3}{2}\pi\nu_j)\Big)\mathrm{d}\mu\ ,
\end{equation}
where $\mathrm{d}\mu=\mathrm{d}\mu_i\mathrm{d}\mu_j$, $i\neq j$ is the standard measure on the hyperplane $\mu_1+\mu_2+\mu_3=0$.

Recall that $\tilde{\mu}$ is fixed with $\|\tilde{\mu}\|= T$ large.

\begin{definition}
For a large $A_0>0$, define $\CH(A_0)$ as the class of test functions $h$ that are Weyl group invariant, holomorphic on $\Lambda_{A_0}$, have rapid decay in this strip, and have zeros at
\begin{equation}
\label{eq:zerosKuznetsov}
    \mu_i-\mu_j=n\equiv 1\mod 2,\quad |n|\leq A_0,\quad 1\leq i<j\leq 3\ .
\end{equation}

We further define $\CH(A_0,\tilde{\mu}):=\CH(A_0,\tilde{\mu},\delta)$ as the set of functions $h\in \CH(A_0)$ such that
\begin{equation}
    \label{eq:rapid_decay_awayfrommu}
    h(\mu)\ll_{A,\varepsilon} (\min_{w\in \mathcal{W}}\|\mu-w(\tilde{\mu})\|)^{-A}
\end{equation}
for $\min_{w\in \mathcal{W}}\|\mu-w(\tilde{\mu})\|\geq T^{\delta+\varepsilon}$, $\varepsilon>0$, and all $A>0$
and \begin{equation}
    \label{eq:bound_h_individuel}
    h(\mu)= O(1)
\end{equation}
for all $\mu$ with $\Re\mu=0$
and
\begin{equation}
\label{eq:bound_derivatives_h}
    D_jh(\mu)\ll_j T^{-j\delta}
\end{equation}
for every differential operator $D_j$ of order $j\geq 0$\ .
\end{definition}

We observe that (\ref{eq:rapid_decay_awayfrommu}) and (\ref{eq:bound_h_individuel}) imply
\begin{equation}
    \label{eq:BoundSpec_h}
    \int_{\Re\mu=0}|h(\mu)\spec(\mu)|\mathrm{d}\mu\ll_{\varepsilon} T^{3+2\delta+\varepsilon}
\end{equation}
for all $\varepsilon>0$.

Lastly, we define $\tilde{\CH}(A_0,\tilde{\mu}):=\tilde{\CH}(A_0,\tilde{\mu},\delta)$ as the set of functions $h\in \CH(A_0,\tilde{\mu},\delta)$ that have additional zeros at
\begin{equation}
\label{eq:zerosLinf}
    \frac{1}{2}+\mu_i+\mu_j=n\equiv 0\mod 2,\quad |n|\leq A_0,\quad 1\leq i\leq j\leq 3\ .
\end{equation}

A typical function we have in mind is
\begin{equation}
    \label{eq:typical_testfunction}
    h(\mu):=\frac{1}{|\mathcal{W}|}\sum_{w\in \mathcal{W}}h_w(\mu):=\frac{1}{|\mathcal{W}|}\sum_{w\in \mathcal{W}}\Big(\Psi(\frac{w(\mu)-\tilde{\mu}}{T^{\delta}})\frac{P(\mu)}{P(\tilde{\mu})}\Big)^2\ ,
\end{equation}
with
\begin{equation*}
    P(\mu):=\Big(\prod_{\substack{|n|\leq A_0\\n\equiv 1\mod 2}}\prod_{1\leq i<j\leq 3}(\mu_i-\mu_j-n)\Big)\Big(\prod_{\substack{|n|\leq A_0\\ n\equiv 0\mod 2}}\prod_{1\leq i\leq j\leq 3}(\frac{1}{2}+ \mu_i+\mu_j-n)\Big)
\end{equation*}
and
\begin{equation*}
    \Psi(\mu):=\exp(\mu_1^2+\mu_2^2+\mu_3^2)\ .
\end{equation*}

A function needs to be in $\CH(A_0)$ to apply the Kuznetsov formula and be able to subsequently shift the $\mu$-contours sufficiently far. The set $\CH(A_0,\tilde{\mu})$ represents functions that are essentially supported in a ball of radius $T^\delta$ around $w(\tilde{\mu})$ for $w\in \mathcal{W}$, are of bounded size in this ball, and do not oscillate too much. The last condition is needed to later bound oscillatory integrals without taking the oscillation of the test function into account. The typical function from above is non-negative and real on $\mu$ with $\Re \mu=0$.

The set of additional zeros in (\ref{eq:zerosLinf}) is needed for the approximate functional equation in Theorem \ref{thm:approxfunctionalequation}.

\subsection{Integral transforms}
\label{sec:inttransforms}
We define the integral kernels as in \cite{BB20}*{Section 3.3}. For $s\in \IC$, $\mu\in \Lambda_\infty$ define the meromorphic function 
\[
    \tilde{G}^{\pm}(s,\mu):=\frac{\pi^{-3s}}{12288\pi^{7/2}}\Big(\prod_{j=1}^3\frac{\Gamma(\frac{1}{2}(s+\mu_j))}{\Gamma(\frac{1}{2}(1-s+\mu_j))}\pm i\prod\frac{\Gamma(\frac{1}{2}(1+s-\mu_j))}{\Gamma(\frac{1}{2}(2-s+\mu_j))}\Big)
\]
and for $s=(s_1,s_2)\in\IC^2$, $\mu\in\Lambda_\infty$ define the meromorphic function
\[
    G(s,\mu):=\frac{1}{\Gamma(s_1+s_2)}\prod_{j=1}^3\Gamma(s_1-\mu_j)\Gamma(s_2+\mu_j)
\]
together with the trigonometric functions
\begin{align*}
    S^{++}(s,\mu):=&\frac{1}{24\pi^2}\cos(\frac{3}{2}\pi\nu_1)\cos(\frac{3}{2}\pi\nu_2)\cos(\frac{3}{2}\pi\nu_3)\\
    S^{+-}(s,\mu):=&-\frac{1}{32\pi^2}\frac{\cos(\frac{3}{2}\pi\nu_2)\sin(\pi(s_1-\mu_1))\sin(\pi(s_2+\mu_2))\sin(\pi(s_2+\mu_3))}{\sin(\frac{3}{2}\pi\nu_1)\sin(\frac{3}{2}\pi\nu_3)\sin(\pi(s_1+s_2))}\\
    S^{-+}(s,\mu):=&-\frac{1}{32\pi^2}\frac{\cos(\frac{3}{2}\pi\nu_1)\sin(\pi(s_1-\mu_1))\sin(\pi(s_1-\mu_2))\sin(\pi(s_2+\mu_3))}{\sin(\frac{3}{2}\pi\nu_2)\sin(\frac{3}{2}\pi\nu_3)\sin(\pi(s_1+s_2))}\\
    S^{--}(s,\mu):=&\frac{1}{32\pi^2}\frac{\cos(\frac{3}{2}\pi\nu_3)\sin(\pi(s_1-\mu_2))\sin(\pi(s_2+\nu_2))}{\sin(\frac{3}{2}\pi\nu_2)\sin(\frac{3}{2}\pi\nu_1)}\ .
\end{align*}

Using these functions, we define the integral kernels as:
\begin{definition}
    For $y\in \IR\setminus\{0\}$ and $\epsilon:=\mathrm{sgn}(y)$, let
    \[
        K_{w_4}(y,\mu):=K_{w_4}^\epsilon(y,\mu):=\int_{(0)}|y|^{-s}\tilde{G}^\epsilon(s,\mu)\frac{\mathrm{d}s}{2\pi i}
    \]
    and for $y=(y_1,y_2)\in (\IR\setminus\{0\})^2$ and $\epsilon_i:=\mathrm{sgn}(y_i)$, let
    \[
        K_{w_6}(y,\mu):= K_{w_6}^{\epsilon_1,\epsilon_2}(y,\mu):=\int_{(0)}\int_{(0)}|4\pi^2y_1|^{-s_1}|4\pi^2y_2|^{-s_2}G(s,\mu)S^{\epsilon_1,\epsilon_2}(s,\mu)\frac{\mathrm{d}s_1\mathrm{d}s_2}{(2\pi i)^2}\ .
    \]
\end{definition}
In this definition and the rest of the paper, we follow the Barnes integral convention that the contour should pass to the right of the poles of Gamma factors of the form $\Gamma(s_j+a)$ and to the left of the poles of Gamma factors of the form $\Gamma(-s_j+a)$. Furthermore, the unbounded part of the contour should be shifted such that  everything is absolutely convergent. This is possible because there is no exponential growth in any of the variables. For $K_{w_4}$, this is obvious, and for $K_{w_6}$, the exponential behavior away from poles was computed in \cite{BB20}*{(3.3)} as $\exp(-\frac{\pi}{2}h^{\epsilon_1,\epsilon_2})(\Im (s),\Im (\mu))$, with
\begin{align}
    \label{eq:def_expbehaviour_w6}
    h^{\epsilon_1,\epsilon_2}(t,r):=-&\epsilon_2|r_1-r_2|-\epsilon_1\epsilon_2|r_1-r_3|-\epsilon_1|r_2-r_3|-\epsilon_1\epsilon_2|t_1+t_2|\\+&\epsilon_1\epsilon_2|t_1-r_1|+\epsilon_1|t_1-r_2|+|t_1-r_3|+|t_2+r_1|+\epsilon_2|t_2+r_2|+\epsilon_1\epsilon_2|t_2+r_3|\ .\nonumber
\end{align}
By brute force, one can check that this is always non-negative. Thus, in both cases, it is enough to shift the unbounded part of the $s$, respectively $s_1$, and $s_2$ contour to $-\delta<0$. 

\begin{definition}
    We define the relevant integral transforms by
    \begin{align*}
        \Phi_{w_4}(y):=\Phi_{w_4}^h(y)&:=\int_{\Re\mu=0}h(\mu)K_{w_4}(y,\mu)\spec(\mu)\mathrm{d}\mu,\\
        \Phi_{w_5}(y):=\Phi_{w_5}^h(y)&:=\int_{\Re\mu=0}h(\mu)K_{w_4}(-y,-\mu)\spec(\mu)\mathrm{d}\mu,\\
        \Phi_{w_6}(y_1,y_2):=\Phi_{w_6}^h(y_1,y_2)&:=\int_{\Re\mu=0}h(\mu)K_{w_6}((y_1,y_2),\mu)\spec(\mu)\mathrm{d}\mu\ .
    \end{align*}
The test function $h$ is implicit in the notation, and we will only specify it if it is not clear from the context.
\end{definition}
\subsection{Kloosterman sums}
For $\GL_3$, there are two relevant types of Kloosterman sums corresponding to the Weyl elements $w_4$ and $w_6$. The Kloosterman sum corresponding to $w_5$ can be expressed using the one corresponding to $w_4$ as in the case of the integral transform.
\begin{definition}
    For $n_1,n_2,m_1,m_2\in\IZ\setminus\{0\}$ and $D_1,D_2\in\IZ^+$, define
    \begin{equation*}
        \tilde{S}(n_1,n_2,m_1;D_1,D_2)=\sum_{\substack{x_1\mod D_1,\ x_2\mod D_2\\ (x_1,D_1)=(x_2,D_2/D_1)=1}}e\Big(n_2\frac{\overline{x}_1x_2}{D_1}+m_1\frac{\overline{x}_2}{D_2/D_1}+n_1\frac{x_1}{D_1}\Big)
    \end{equation*}
    for $D_1\mid D_2$, and
    \begin{align*}
        &S(n_1,n_2,m_1,m_2;D_1,D_2)\\=&\sum_{\substack{B_1,C_1\mod D_1\\ B_2,C_2\mod D_2\\ D_1C_2+B_1B_2+D_2C_1\equiv 0 \mod D_1D_2\\ (B_j,C_j,D_j)=1}}e\Big(\frac{n_1B_1+m_1(Y_1D_2-Z_1B_2)}{D_1}+\frac{n_2B_2+m_2(Y_2D_1-Z_2B_1)}{D_2}\Big)\ ,
    \end{align*}
    where $Y_jB_j+Z_jC_j\equiv 1 \mod D_j$.
\end{definition}

We quote \cite{BFG88}*{Appendix}, \cite{St87}*{Section 5}, and \cite{KN22}*{Theorem 5.10}.
\begin{lemma}

\label{lemma:b_f_KLsums}
\begin{enumerate}[(a)]
    \item We have
    \[
    \tilde{S}(n_1,n_2,m_1)\ll ((m_1,D_2/D_1)D_1^2,(n_1,n_2,D_1)D_2)(D_1D_2)^\varepsilon\ .
    \]
    \item We have
    \[
    S(n_1,n_2,m_1,m_2;D_1,D_2)\ll (D_1D_2)^{1/2+\varepsilon}((D_1,D_2)(m_1n_1,[D_1,D_2])(m_2n_2,[D_1,D_2]))^{1/2}\ .
    \]
    \item We also have the parametrization
    \[
        S(n_1,n_2,m_1,m_2;D_1,D_2)=\sum_{D_0\mid (D_1,D_2)}S(n_1,n_2,m_1,m_2;\frac{D_1}{D_0},\frac{D_2}{D_0},D_0)\ ,
    \]
    with the fine Kloosterman sum $S(n_1,n_2,m_1,m_2;D_1,D_2,D_0)$ defined by
    \begin{align*}
        &S(n_1,n_2,m_1,m_2;D_1,D_2,D_0)\\= &D_0\sum_{\substack{x,y\mod D_0\\ xy\equiv 1\mod D_0\\ m_1D_2+n_2D_1y\equiv 0\mod D_0}}S(n_1,\frac{m_1D_2+n_2D_1y}{D_0};D_1)S(m_2,\frac{m_1D_2x+n_2D_1}{D_0};D_2)\ .
    \end{align*}
\end{enumerate}
\end{lemma}
\subsection{The Kuznetsov formula}
We state the Kuznetsov formula as in \cite{Bu16}*{Theorem 4} and \cite{BB20}*{Section 3.4}.

\begin{lemma}[The Kuznetsov Formula]
    Let $h\in \CH(A_0)$. Then we have
    \[
        \int_\pi \overline{A_\pi(m_1,m_2)}A_\pi(n_1,n_2)\frac{h(\mu_\pi)}{\CN(\pi)}\mathrm{d}\pi=\Delta+\Sigma_4+\Sigma_5+\Sigma_6\ ,
    \]
    where
    \begin{align*}
        \Delta&=\delta_{n_1,m_1}\delta_{n_2,m_2}\frac{1}{192\pi^5}\int_{\Re\mu=0}h(\mu)\spec(\mu)\mathrm{d}\mu\ ,\\
        \Sigma_4&=\sum_{\epsilon=\pm}\sum_{\substack{D_2\mid D_1\\m_2D_1=n_1D_2^2}}\frac{\tilde{S}(-\epsilon n_2,m_2,m_1;D_2,D_1)}{D_1D_2}\Phi_{w_4}(\frac{\epsilon m_1m_2n_2}{D_1D_2})\ ,\\
        \Sigma_5&=\sum_{\epsilon=\pm}\sum_{\substack{D_1\mid D_2\\m_1D_2=n_2D_1^2}}\frac{\tilde{S}(\epsilon n_1,m_1,m_2;D_1,D_2)}{D_1D_2}\Phi_{w_5}(\frac{\epsilon m_1m_2n_1}{D_1D_2})\ ,\\
        \Sigma_6&= \sum_{\epsilon_1\epsilon_2=\pm}\sum_{D_1,D_2}\frac{S(\epsilon_2 n_2,m_2,m_1,\epsilon_1 n_1;D_1,D_2)}{D_1D_2}\Phi_{w_6}(-\frac{\epsilon_2 m_1n_2D_2}{D_1^2},-\frac{\epsilon_1 m_2n_1D_1}{D_2^2})
    \end{align*}
\end{lemma}
and the integral over $\pi$ is interpreted as a sum over the discrete part of the spectrum and an integral over the continuous part. For more information regarding the spectral decomposition, we refer to \cite{Bu16}.

\subsection{Rough bounds and absolute convergence}
\label{sec:roughboundsabsconv}
We usually want to truncate the sums over $D_1,D_2$ at $D_1,D_2\ll T^B$ for some sufficiently large $B$ at the cost of a negligible error, provided that $n_1,n_2,m_1,m_2\ll T^{C}$ for some fixed $C>0$. To achieve this, we need the following rough bound on the integral transforms.
\begin{lemma}[{\cite{BB20}*{Lemma 1}}]
\label{lemma:rough_bound_inttransform}
With $h\in \CH(A_0)$ also having rapid decay away from $\tilde{\mu}$ (\ref{eq:rapid_decay_awayfrommu}) and bounded spectral norm (\ref{eq:BoundSpec_h}) we have
\[
    \Phi_{w_4}(y)\ll |y|^{1/10}T^{O(1)},\quad \Phi_{w_5}(y)\ll |y|^{1/10}T^{O(1)},\quad  \Phi_{w_6}(y)\ll |y_1y_2|^{3/5}T^{O(1)}\ .
\]
\end{lemma}
If we assume that $m_1,m_2,n_1,n_2$ are bounded by $T^{O(1)}$, then together with the bounds from Lemma \ref{lemma:b_f_KLsums} (a) and (b), we see that the sums over Kloosterman terms in $\Sigma_4,\Sigma_5,\Sigma_6$ are absolutely convergent and that we can truncate the $D_1,D_2$-sums as we wanted at $D_1,D_2\ll T^B$.

\section{Analytic Preliminaries}
In this section, we compile some frequently used analytic results for future reference.

\subsection{Oscillatory integrals}
Our main tool for dealing with oscillatory integrals will be integration by parts. The following two Lemmata make this explicit.
\begin{lemma}[{\cite{BB20}*{Lemma 2}}]
\label{lemma:intbyparts}
Let $Y\geq 1,X,Q,U,R>0$, and suppose that $w$ is a smooth function with support on some interval $[\alpha,\beta]$ satisfying
\[
    w^{(j)}(t)\ll_j XU^{-j}\ .
\]
Also suppose $H$ is a smooth function on $[\alpha,\beta]$ such that 
\[
    H^\prime(t)\gg R,\quad H^{(j)}(t)\ll_jYQ^{-j} \textnormal{ for }j\geq 2\ .
\]
Then
\[
    I=\int_\IR w(t)\exp(i H(t))\mathrm{d}t\ll_B (\beta-\alpha)X\big((QR/\sqrt{Y})^{-B}+(RU)^{-B}\big)
\]
for any $B>0$.
\end{lemma}

\begin{lemma}[{\cite{BB20}*{Lemma 3}}]
\label{lemma:intbyparts_asymptotic}
    Let $x,t\in \IR$, and let $w$ be a fixed smooth function with compact support on $\IR^+$. Let 
    \[
        I:=\int_{\IR^+}w(y)e(xy)y^{-it}\mathrm{d}y
    \]
    and let $B>0$. There exists a smooth function $\tilde{w}_x(y)$ with compact support on $\IR^+$ satisfying $\tilde{w}_x^{(j)}(y)\ll_j 1$ with the following property: If $|x|+|t|\geq 100$, then
    \[
        I=|x|^{-1/2}\Big|\frac{t}{2\pi ex}\Big|^{-it}\tilde{w}_x\Big(\frac{t}{x}\Big)+O_B((|x|+|t|)^{-B})\ .
    \]
\end{lemma}
\begin{rem}
    Although $\tilde{w}_x$ does depend on $x$, the bounds on its derivatives and the size of the support do not. 
\end{rem}

\subsection{Stirling's formula}
To treat the Gamma factors, we will mainly use Stirling's formula for the Gamma function. For $\sigma\in\IR$ fixed, real $t$ with $|t|\geq 10$ and any $B>0$, we have
\begin{equation}
    \label{eq:stirling}
    \Gamma(\sigma+it)=|t|^{\sigma-\frac{1}{2}}\exp\Big(-\frac{\pi}{2}|t|+it\log\frac{|t|}{e}\Big)g_{\sigma,B}+O_{\sigma,B}(\exp(-\frac{\pi}{2}|t|)|t|^{-B})\ ,
\end{equation}
with
\[
    |t|^j\frac{\partial^j}{\partial t^j}g_{\sigma,B}\ll_{j,\sigma,B}1
\]
for $j\geq 0$.

This will mainly be used to show that the exponential part of a product of Gamma factors is decaying or constant, and to estimate the polynomial growths correctly.

We will also use the Digamma function $\psi$ defined as
\[
    \psi(z):=\frac{\Gamma^\prime(z)}{\Gamma(z)}\ .
\]
For $\mathrm{arg}(z)=\pi-\varepsilon^\prime$, it satisfies for any $B>0$
\begin{equation}
    \label{eq:stirling_digamma}
    \psi(z)=\log(z)+g_B(z)+O_{B,\varepsilon^\prime}(|z|^{-B})\ ,
\end{equation}
with $g_B$ a rational function satisfying
\[
    |z|^j\frac{\partial^j}{\partial z^j}g_B(z)\ll_{j} |z|^{-1}
\]
for $j\geq 0$. And also
\begin{equation}
    \label{eq:bound_derivative_digamma}
    |z|^j\frac{\partial^j}{\partial z^j}(\psi(z)-\log(z))\ll_{j,\varepsilon^\prime} |z|^{-1}
\end{equation}
for all $j\geq 0$.

\subsection{Bessel functions and integral representations} The integral transforms in the Kuznetsov formula can alternatively be expressed using Bessel functions, similarly to the $\GL_2$ setting. Such integral representations and results about the relevant Bessel functions are compiled in 
\cite{BB20}*{Section 4 and 5}. In this subsection, we recall some of these that are later needed to prove Lemma \ref{lemma:fine_bound_inttransfrom_voronoi} and Lemma \ref{lemma:fine_bound_inttransfrom_long}.

We define the following functions for $x>0,\alpha\in \IC$:
\[
    J^+_\alpha(x)=\frac{\pi}{2}\frac{J_{-\alpha}(2x)+J_{\alpha}(2x)}{\cos(2\pi\alpha/2)},\ J^-_\alpha(x)=\frac{\pi}{2}\frac{J_{-\alpha}(2x)-J_{\alpha}(2x)}{\sin(2\pi\alpha/2)}\ ,\tilde{K}_\alpha(x)=2\cos(\frac{\pi}{2}\alpha)K_a(2x)\ ,
\]
where $J_\alpha$ and $K_\alpha$ are the usual Bessel functions. By \cite{BB20}*{(4.12),(4.13),(4.17),(4.20)}, these functions satisfy:

\begin{equation}
    \label{eq:4.12}
    \tilde{K}_{it}(x)=\int_\IR \cos(2x\sinh v)\exp(itv)\mathrm{d}v,\quad J^\pm_{it}(x)=\int_\IR \left\{
\begin{array}{c}
\cos(x) \\
\sin(x)
\end{array}
\right\} (2x\cosh v)\exp(itv)\mathrm{d}v
\end{equation}
for $t\in \IR,x>0$. The integrals are not absolutely convergent, but integration by parts shows that the tail is very small, so the conditional convergence causes no extra difficulty. We have
\begin{equation}
    \label{eq:4.13}
    \frac{\partial^j}{\partial x^j}\tilde{K}_{it}(x),\frac{\partial^j}{\partial x^j} J^\pm_{it}(x)\textbf{}\ll_j (1+\frac{|t|}{x})^j
\end{equation}
for $|t|,x\geq 1$ and $j\geq 0$.

We also have
\begin{equation}
    \label{eq:4.17}
    \tilde{K}_{it}(x/2)=\Re(\exp(i\omega (x,t))f_M(x,t))+O(|t|^{-M}),\quad \omega(x,t)=|t|\mathrm{arccosh}\frac{|t|}{x}-\sqrt{t^2-x^2}\ ,
\end{equation}
for $t\in\IR,|t|>1,\frac{1}{10}|t|\geq x>0$ and fixed $M>0$ with
\[  
    x^i|t|^j\frac{\partial^j}{\partial x^i}\frac{\partial^j}{\partial t^j}f_M(x,t)\ll_{i,j,M}|t|^{-1/2}
\]
for any $i,j\geq 0$.

And similarly
\begin{equation}
\label{eq:4.20}
    J^\pm_{it}(x/2)=\Re(\exp(i\tilde{\omega}(x,t))\tilde{f}^\pm_M(x,t))+O(|t|^{-M}),\quad \tilde{\omega}(x,t)=|t|\mathrm{arcsinh}\frac{|t|}{x}-\sqrt{t^2+x^2}\ ,
\end{equation}
for $t\in\IR,|t|>1, x>0$ and fixed $M>0$ with
\[  
    x^i|t|^j\frac{\partial^j}{\partial x^i}\frac{\partial^j}{\partial t^j}\tilde{f}_M(x,t)\ll_{i,j,M}\frac{1}{x^{1/2}+|t|^{1/2}}
\]
for any $i,j\geq 0$.

Lastly, we quote an integral representation of $K_{w_4}$ and of $K_{w_6}$ in the $(++)$ case.

By \cite{BB20}*{Lemma 4}, we have
\begin{align}
    \label{eq:lemma4}
    K_{w_4}(y,\mu)=&\frac{1}{3072\pi^5}\int_0^\infty J^-_{\mu_1-\mu_2}(2\sqrt{u})\Big(\frac{\pi^3|y|}{u^{3/2}}\Big)^{-\mu_3}\exp(-\frac{2 i\pi^3 y}{u})\frac{\mathrm{d}u}{u}\\
    +&\frac{1}{3072\pi^5}\int_0^\infty \tilde{K}_{\mu_1-\mu_2}(2\sqrt{u})\Big(\frac{\pi^3|y|}{u^{3/2}}\Big)^{-\mu_3}\exp(-\frac{2 i\pi^3 y}{u})\frac{\mathrm{d}u}{u}\ .\nonumber
\end{align}
The integrals barely fail to be absolutely convergent at $0$, but since $\exp(\pm2 i\pi^3 y/u)$
is highly oscillating in a neighborhood of $u=0$, the integrals exist in a Riemann sense, and
the portion $0 < u < 1$ can be made absolutely convergent after partial integration.

By \cite{BB20}*{Lemma 5 and (5.7)} for $y_1,y_2>0$, we have
\begin{equation}
    \label{eq:5.8}
    K_{w_6}(y,\mu)=\frac{1}{12\pi^2}\frac{\cos(\frac{3}{2}\pi\nu_1)\cos(\frac{3}{2}\pi\nu_2)}{\cos(\frac{3}{2}\pi\nu_3)}\mathcal{J}_5(y,\mu)
\end{equation}
with
\begin{equation}
    \label{eq:5.7}
    \mathcal{J}_5(y,\mu)=\left|\frac{y_1}{y_2}\right|^{\frac{1}{2}\mu_2}\int_0^\infty \tilde{K}_{3\nu_3}(2\pi |y_1|^{1/2}\sqrt{1+u^2})\tilde{K}_{3\nu_3}(2\pi |y_2|^{1/2}\sqrt{1+u^{-2}})u^{3\mu_2}\frac{\mathrm{d}u}{u}\ .
\end{equation}

\section{Approximate functional equation}
To access the central L-value, we are using an asymmetric approximate functional equation similar to \cite{BC25}*{Section 6.1}. There are two important features compared to a standard approximate functional equation as in \cite{IK04}*{Section 5.2}. The first is the addition of a polynomial $G$ that ensures the poles of the $L$-factor at infinity are canceled in a wide strip. This is important for later being able to shift the $\mu$-contours. The second is the asymmetry between the two parts. This is useful because the second part contains an oscillating factor coming from the $L$-factor at infinity that makes later arguments harder. The slight imbalance between the two summands enables us to treat the second summand more bluntly and avoid dealing with the added oscillation in some cases.

Define
\[
    \mathcal{G}(u,\mu):=(1-4u^2)\prod_{\substack{n\equiv 0 \mod 2\\|n|\leq A_0}}\prod_{1\leq i\leq j\leq 3}(\frac{1}{2}- u+ \mu_i+\mu_j-n)\ .
\]

\begin{thm}
\label{thm:approxfunctionalequation}
For $\pi\in\CB$ we have
\[
    L(1/2,\pi,\mathrm{sym}^2)=\zeta(\frac{3}{2})\cdot \sum_{b,c}\Big( \frac{A_\pi(b^2,c^2)}{(b^2c)^{1/2}}V_\pi(\frac{b^2c}{T^{3+\alpha}})-\frac{A_{\Tilde\pi}(b^2,c^2)}{(b^2c)^{1/2}}W_\pi(b^2cT^{3+\alpha})\Big)\ ,
\]
with
\[
    V_\pi(x):=V_{\mu_\pi}(x):=\int_{(2)} x^{-u}\frac{e^{u^2}}{u}\frac{\zeta(3(\frac{1}{2}+u))}{\zeta(\frac{3}{2})}\frac{\mathcal{G}(u,\mu_\pi)}{\mathcal{G}(0,\mu_\pi)}\frac{\mathrm{d}u}{2\pi i}
\]
and
\[
    W_\pi(x):=W_{\mu_\pi}(x):=\int_{(2)} x^{-u}\frac{e^{u^2}}{u}\frac{\zeta(3(\frac{1}{2}+u))}{\zeta(\frac{3}{2})}\frac{\mathcal{G}(-u,\mu_\pi)}{\mathcal{G}(0,\mu_\pi)}\prod_{1\leq i\leq j\leq 3}\frac{\Gamma_{\IR}(\frac{1}{2}+u+\mu_i+\mu_j)}{\Gamma_{\IR}(\frac{1}{2}-u-\mu_i-\mu_j)}\frac{\mathrm{d}u}{2\pi i}\ .
\]
\end{thm}
\begin{proof}
    We consider the integral:
\[
    \int_{(2)} L(u+1/2,\pi,\mathrm{sym}^2)\frac{e^{u^2}}{u} \mathcal{G}(u,\mu_\pi)T^{(3+\alpha)u}\frac{\mathrm{d}u}{(2\pi i)}\ .
\]
We first shift the contour to the left and pick up the residue
\[
    L(1/2,\pi,\mathrm{sym}^2)\mathcal{G}(0,\mu_\pi)
\]
at $u=0$. By Lemma \ref{lemma:lfuncisentire}, there are no other poles except for potential simple poles at $u=\pm \frac{1}{2}$, but these are canceled by the zeros of $G$. 

We then apply the functional equation \ref{eq:funcequation}
to the remaining integral and divide by $\mathcal{G}(0,\mu_\pi)$, which is nonzero due to Lemma \ref{lemma:lfuncisentire}, and change variables from $u$ to $-u$. We also use that by unitarity (\ref{eq:unitarity}), we have $\{\mu_1,\mu_2,\mu_3\}=\{-\overline{\mu_1},-\overline{\mu_2},-\overline{\mu_3}\}$. Lastly, we insert the Dirichlet series for $L$, as the argument has real part greater than one, and pull the sum and $\zeta(\frac{3}{2})$ out of the integral.
\end{proof}
Notice that $\mathcal{G}(-u,\mu_\pi)$ cancels all the poles of the Gamma factors in a wide strip. The factor $\frac{\zeta(3(\frac{1}{2}+u))}{\zeta(\frac{3}{2})}$ is holomorphic and bounded from above and below, with bounds depending only on the real part of $u$. It will, therefore, not affect the future computations and is sometimes implicitly bounded by a constant depending on the real part of $u$.
\begin{rem}
    \label{rem:selfdualforms}
    The (in)balance between the two terms containing $V$ and $W$ is defined in such a way that the length of the sum for the $W$ will later be slightly smaller than the critical range, which enables us to discard it in many cases. This critically uses 
    \[
      \prod_{1\leq i\leq j\leq 3}\frac{\Gamma_{\IR}(\frac{1}{2}+u+\mu_i+\mu_j)}{\Gamma_{\IR}(\frac{1}{2}-u-\mu_i-\mu_j)}\asymp T^{6\Re u} 
    \]
    for $\|\mu\|\asymp T$ in generic position and $\Im (u)\ll T^\varepsilon$. If we let $\pi$ be near a self-dual form, one of the spectral parameters will be close to $0$ and the analytic conductor, and thus also 
    \[
      \prod_{1\leq i\leq j\leq 3}\frac{\Gamma_{\IR}(\frac{1}{2}+u+\mu_i+\mu_j)}{\Gamma_{\IR}(\frac{1}{2}-u-\mu_i-\mu_j)} 
    \]
    will be smaller. This would require us to replace $T^3$ in the approximate functional equation by $\prod_{i=1}^3(|\mu_i|+1)$. Theorem \ref{thm:MainTheorem} should still hold with the same main and error term, but the methods later in the proof would need to distinguish a lot more cases, especially in Section \ref{section:LemmaWeightFunctions}, as the size of the $\mu_i$ is critically used there.
\end{rem}

The two functions $V,W$ are holomorphic in $\mu$ for $\Re(\mu)<A_0$ and satisfy the following bounds:

\begin{lemma}
\label{lemma:bound_VW}
    For $x\rightarrow\infty$ and $\mu$ sufficiently large and in generic position, we have
    \[
        V_\mu(x)\ll_A x^{-A}\text{ and } W_\mu(x)\ll_A x^{-A} \|\mu\|^{6A} \ ,
    \]
    for $A<A_0$.
\end{lemma}
\begin{proof}
    Shift the contour to $A$. This needs $A_0>A$. We have
    \[
        \int_\IR e^{-x^2}|P(x)|\ll_{\mathrm{deg}(P)}\max(|a_i|)
    \]
    for any polynomial $P$ with coefficients $a_i$. Thus:
    \[
        V_{\mu}(x)\ll_A x^{-A}\int_{(A)}|e^{u^2}/u||\mathcal{G}(u,\mu)|\mathrm{d}u |\mathcal{G}(0,\mu)|^{-1}\ll_A x^{-A}
    \]
    because the coefficients are bounded by $\|\mu\|^{\mathrm{deg}(G)}$ and $|\mathcal{G}(0,\mu)| \asymp \|\mu\|^{\mathrm{deg}(G)}$ for $\mu$ large enough and in generic position.
    
     The proof for $W$ is the same with the added bound
    \[
        \Big|\prod_{1\leq i\leq j\leq 3}\frac{\Gamma_{\IR}(\frac{1}{2}+u+\mu_i+\mu_j)}{\Gamma_{\IR}(\frac{1}{2}-u-\mu_i-\mu_j)}\Big| \ll_u \|\mu\|^{6 \Re(u)}
    \]
    with polynomial dependency on $u$ that follows from Stirling's formula (\ref{eq:stirling}).
\end{proof}

To later perform integration by parts, we need the following lemma.

\begin{lemma}
\label{lemma:intpartsVW}
    For $\mu$ in generic position, $|u|\ll \|\mu\|^{\varepsilon}$ and any differential operator $D_k$ in $\mu$ of order $k$, we have
    \[
        D_k\frac{\mathcal{G}(u,\mu)}{\mathcal{G}(0,\mu)}\ll_u \|\mu\|^{-k}\left|\frac{\mathcal{G}(u,\mu)}{\mathcal{G}(0,\mu)}\right| 
    \]
    with polynomial dependency on $u$.
    
     For $\mu$ in generic position, $|u|\ll \|\mu\|^{\varepsilon}$ and any differential operator $D_k$ of order $k$ in $\mu$, we have
    \begin{align*}
        &D_k\Big(\frac{\mathcal{G}(-u,\mu)}{\mathcal{G}(0,\mu)}\prod_{1\leq i\leq j\leq 3}\frac{\Gamma_{\IR}(\frac{1}{2}+u+\mu_i+\mu_j)}{\Gamma_{\IR}(\frac{1}{2}-u-\mu_i-\mu_j)}\exp(-h_\infty(\mu))\Big)\\
        \ll_{u,\varepsilon}&  \|\mu\|^{-k}\Big|\frac{\mathcal{G}(-u,\mu)}{\mathcal{G}(0,\mu)}\prod_{1\leq i\leq j\leq 3}\frac{\Gamma_{\IR}(\frac{1}{2}+u+\mu_i+\mu_j)}{\Gamma_{\IR}(\frac{1}{2}-u-\mu_i-\mu_j)}\exp(-h_\infty(\mu))\Big|
    \end{align*}
    with
    \[
        h_\infty(\mu):= \sum_{1\leq i\leq j\leq 3} (\mu_i+\mu_j)\log(\frac{|\mu_i+\mu_j|}{2e})
    \]
    and polynomial dependency on $u$.
\end{lemma}

The function $h_\infty(\mu)$ basically captures the phase of the Gamma factors up to lower order terms, ignoring the imaginary part of $u$, which is not an issue because $|\Im(u)|\ll\|\mu\|^\varepsilon$.

As it is often useful to parametrize the $\mu$-plane by $\mu_i,\nu_{i+1}$, we also have this alternative way to write $h_\infty$:
\begin{equation}
    \label{eq:alternative_h_inf}
    h_\infty(\mu)= \frac{1}{2}\big(\mu_i\cdot \log\big(\frac{|\mu_i|^2}{|3\nu_{i+1}+\mu_i||3\nu_{i+1}-\mu_i|}\big) +3\nu_{i+1}\log\big(\frac{|3\nu_{i+1}-\mu_i|}{|3\nu_{i+1}+\mu_i|}\big)\big) 
\end{equation}
for $1\leq i\leq 3$ and $\nu_4:=\nu_1$.
We can already see that the quotients in the logarithms are basically constant if $\mu$ is close to $\tilde{\mu}$, which will be used later. 

\begin{proof}
The function $\frac{\mathcal{G}(u,\mu)}{\mathcal{G}(0,\mu)}$ is rational in the $\mu_i$ and $u$. Thus, using that for each $i$, we have $|\mu_i|\asymp \|\mu\|$ leads to the first statement.

For the second statement, the same is true for the $\frac{\mathcal{G}(-u,\mu)}{\mathcal{G}(0,\mu)}$ part. Using $\mu_1+\mu_2+\mu_3=0$, we also notice
\[
    \prod_{1\leq i\leq j\leq 3}\frac{\Gamma_{\IR}(\frac{1}{2}+u+\mu_i+\mu_j)}{\Gamma_{\IR}(\frac{1}{2}-u-\mu_i-\mu_j)}=\prod_{1\leq i\leq j\leq 3}\frac{\Gamma(\frac{\frac{1}{2}+u+\mu_i+\mu_j}{2})}{\Gamma(\frac{\frac{1}{2}-u-\mu_i-\mu_j}{2})}\pi^{-6u} \ .
\]
Using this, we have for any differential operator $D$ in the $\mu_i$ of order $1$
\begin{align*}
    &D\Big(\prod_{1\leq i\leq j\leq 3}\frac{\Gamma_{\IR}(\frac{1}{2}+u+\mu_i+\mu_j)}{\Gamma_{\IR}(\frac{1}{2}-u-\mu_i-\mu_j)}\exp(-h_\infty(\mu))\Big)\\
    =&\Big(-D(h_\infty(\mu))+\sum_{1\leq i\leq j\leq 3}D(\frac{\frac{1}{2}+u+\mu_i+\mu_j}{2})\cdot \psi(\frac{\frac{1}{2}+u+\mu_i+\mu_j}{2})\\
    &-D(\frac{\frac{1}{2}-u-\mu_i-\mu_j}{2})\cdot \psi(\frac{\frac{1}{2}-u-\mu_i-\mu_j}{2})
    \Big)\cdot \prod_{1\leq i\leq j\leq 3}\frac{\Gamma_{\IR}(\frac{1}{2}+u+\mu_i+\mu_j)}{\Gamma_{\IR}(\frac{1}{2}-u-\mu_i-\mu_j)}\exp(-h_\infty(\mu)) \ .
\end{align*}
Because $|u|\ll \|\mu\|^\varepsilon$, we can use (\ref{eq:bound_derivative_digamma}) with $\varepsilon^\prime=\pi/4$. This gives
\begin{align*}
    &-D(h_\infty(\mu))+\sum_{1\leq i\leq j\leq 3}D(\frac{\frac{1}{2}+u+\mu_i+\mu_j}{2})\cdot \psi(\frac{\frac{1}{2}+u+\mu_i+\mu_j}{2})\\
    &\quad\quad-D(\frac{\frac{1}{2}-u-\mu_i-\mu_j}{2})\cdot \psi(\frac{\frac{1}{2}-u-\mu_i-\mu_j}{2})\\
    &=\sum_{1\leq i\leq j\leq 3}D(\frac{\mu_i+\mu_j}{2})\cdot \log(\frac{1}{4}(\frac{1}{2}+u+\mu_i+\mu_j)(\frac{1}{2}-u-\mu_i-\mu_j))\\&\quad\quad-D(\mu_i+\mu_j)\cdot \log(\frac{1}{2}|\mu_i+\mu_j|)+\tilde{g}_{\varepsilon^\prime}(u,\mu)\\
    &=\sum_{1\leq i\leq j\leq 3}D(\frac{\mu_i+\mu_j}{2})\log\big(\frac{(\frac{1}{2}+u+\mu_i+\mu_j)(\frac{1}{2}-u-\mu_i-\mu_j)}{|\mu_i+\mu_j|^2}\big)+\tilde{g}_{\varepsilon^\prime}(u,\mu),
\end{align*}
where $\tilde{g}_{\varepsilon^\prime}(u,\mu)$ is polynomially bounded in $u$ and satisfies
\[
\frac{\partial^j}{\partial \mu^j}\tilde{g}_{\varepsilon}(u,\mu)\ll_j \|\mu\|^{-j-1}  
\]
for $j\geq 0$.

Looking at the quotient in the logarithm, we can see that it is equal to $1+O(|u|\cdot \|\mu\|^{-1})$. Thus, the logarithm is bounded by $|u|\cdot \|\mu\|^{-1}$, which proves the claim for a differential operator of order $1$. For higher order operators, we can inductively repeat the argument, using that the additional factor was not only bounded by $\|\mu\|^{-1}$, but also its $j$-th derivative is bounded by $\|\mu\|^{-j-1}$.
\end{proof}

\begin{lemma}
\label{lemma:derivative_hinf}
The following statements hold for the derivatives of $h_\infty$ for $\mu\in\Lambda_0$ with $\|\mu-\tilde{\mu}\|\leq T^{\delta}$:
\begin{align*}
    \frac{\partial}{\partial\Im(\mu_i)}h_\infty(\mu)=&\frac{i}{2}\log\big|\frac{1}{4}\frac{\tilde{\mu}_i^2}{\tilde{\mu}_{i+1}\tilde{\mu}_{i+2}}\big|+O(T^{-1+\delta})\\
    \frac{\partial}{\partial\Im(\nu_{i+1})}h_\infty(\mu)=&\frac{3i}{2}\log\big|\frac{\tilde{\mu}_{i+1}}{\tilde{\mu}_{i+2}}\big|+O(T^{-1+\delta})\\
    \frac{\partial^j}{\partial\Im(\mu_{i})^j}h_\infty(\mu),\frac{\partial^j}{\partial\Im(\nu_{i+1})^j}h_\infty(\mu)\ll& T^{1-j}
\end{align*}
for all $j\geq 2$.
\end{lemma}
\begin{proof}
Recall (\ref{eq:alternative_h_inf})
\[
    h_\infty(\mu)= \frac{1}{2}\big(\mu_i\cdot \log\big(\frac{|\mu_i|^2}{|3\nu_{i+1}+\mu_i||3\nu_{i+1}-\mu_i|}\big) +3\nu_{i+1}\log\big(\frac{|3\nu_{i+1}-\mu_i|}{|3\nu_{i+1}+\mu_i|}\big)\big) \ .
\]
We calculate
\begin{align*}
    \frac{\partial}{\partial\Im(\mu_i)}h_\infty(\mu)=&\ \frac{i}{2}\Big(\log\big(\frac{|\mu_i|^2}{|3\nu_{i+1}+\mu_i||3\nu_{i+1}-\mu_i|}\big)+2-\frac{\Im(\mu_i)}{\Im(3\nu_{i+1}+\mu_i)}+\frac{\Im(\mu_i)}{\Im(3\nu_{i+1}-\mu_i)}\\
    &\ + \frac{\Im(3\nu_{i+1})}{\Im(-3\nu_{i+1}+\mu_i)}-\frac{\Im(3\nu_{i+1)}}{\Im(3\nu_{i+1}+\mu_i)}
    \Big)\\
    =&\ \frac{i}{2}\log\big(\frac{|\mu_i|^2}{|3\nu_{i+1}+\mu_i||3\nu_{i+1}-\mu_i|}\big)\\
    =&\ \frac{i}{2}\log\big|\frac{1}{4}\frac{\mu_i^2}{\mu_{i+1}\mu_{i+2}}\big|= \frac{i}{2}\log\big|\frac{1}{4}\frac{\tilde{\mu}_i^2}{\tilde{\mu}_{i+1}\tilde{\mu}_{i+2}}\big|+O(T^{-1+\delta})\ .
\end{align*}
In the last step, we can use that $\mu$ is close to $\tilde{\mu}$, and we can pull the error term outside the logarithm because
\[
0<\frac{c_1^2}{c_2^2}\leq \big|\frac{1}{4}\frac{\tilde{\mu}_i^2}{\tilde{\mu}_{i+1}\tilde{\mu}_{i+2}}\big|\leq \frac{c_2^2}{c_1^2}\ .
\]
We calculate further
\[
    \frac{\partial^2}{\partial\Im(\mu_i)^2}h_\infty(\mu)=\ \frac{i}{2}\big(\frac{2}{\Im(\mu_i)}-\frac{1}{\Im(3\nu_{i+1}+\mu_i)}+\frac{1}{\Im(-3\nu_{i+1}+\mu_i)}\big)\ .
\]
This is in $O(T^{-1})$ and rational in $\mu_i$, which implies the above claim for all $j\geq2$.

For $\nu_{i+1}$, we calculate in the same way
\begin{align*}
    \frac{\partial}{\partial\Im(\nu_{i+1})}h_\infty(\mu)=&\ \frac{3i}{2}\log\big(\frac{|3\nu_{i+1}-\mu_i|}{|3\nu_{i+1}+\mu_i|}\big)\\
    =&\ \frac{3i}{2}\log\big|\frac{\mu_{i+1}}{\mu_{i+2}}\big|\\
    =&\ \frac{3i}{2}\log\big|\frac{\tilde{\mu}_{i+1}}{\tilde{\mu}_{i+2}}\big|+O(T^{-1+\delta})\ ,
\end{align*}
where we used the same argument in the last step, but with
\[
    \frac{c_1}{c_2}\leq \frac{|\tilde{\mu}_{i+1}|}{|\tilde{\mu}_{i+2}|}\leq \frac{c_2}{c_1}\ .
\]  
Lastly, we have
\begin{align*}
    \frac{\partial^2}{\partial\Im(\nu_{i+1})^2}h_\infty(\mu)=\ \frac{9i}{2}\Big(\frac{1}{\Im(3\nu_{i+1}-\mu_i)}-\frac{1}{\Im(3\nu_{i+1}+\mu_i)}\Big)\ .
\end{align*}
This is in $O(T^{-1})$ and rational in $\nu_{i+1}$, which implies the above claim for all $j\geq2$.
\end{proof}

\section{Key Lemmata on the weight functions}
\label{section:LemmaWeightFunctions}
In this section, we compile bounds on the integral transforms $\Phi_w$ and some Fourier transforms of them. The results serve two main purposes. Lemma \ref{lemma:fine_bound_inttransfrom_voronoi} and Lemma \ref{lemma:fine_bound_inttransfrom_long} provide finer bounds for $\Phi_w$ for different sizes of the arguments. Most importantly, they quantify that $\Phi_w(y)$ is negligible for $y$ small. Lemma \ref{lemma:bound_fourier_voronoi} and Lemma \ref{lemma:bound_fourier_long} give strong bounds on Fourier transforms of $\Phi_w$. These show up after applying Poisson summation in the proof of Theorem \ref{thm:MainTheorem} and are the main technical input of this paper.

All four lemmata  have a similar structure and purpose as the results in \cite{BB20}*{Section 7}, and their proofs follow similar ideas but differ in two key aspects. The first is that our test functions include terms coming from the approximate functional equation, Theorem \ref{thm:approxfunctionalequation}, in particular, the additional quotient of Gamma-factors contained in $W_\mu(x)$. The second is that we will later perform Poisson summation in only one variable and thus will encounter different Fourier transforms compared to \cite{BB20}.

In all the upcoming lemmata, we will distinguish between two kinds of test functions. One will later correspond to the $V$-part of the approximate functional equation, and one will correspond to the $W$-part. The latter has added oscillation captured by $h_\infty$; see Lemma \ref{lemma:intpartsVW} and (\ref{eq:alternative_h_inf}). As the main tools to bound the integral transforms are stationary phase or integration by parts, added oscillation complicates the arguments. However, as seen below, we can often overcome this because $h_\infty$ is essentially linear in the $\mu_i$ and $\nu_i$.

\begin{lemma}
\label{lemma:fine_bound_inttransfrom_voronoi}
Let $h\in \CH(A_0,\tilde{\mu})$.
\begin{enumerate}[(a)]
    \item For $0<|y|\leq T^{3-\varepsilon}$ and any $B\geq 0$, we have
    \[
        \Phi_{w_4}^h(y)\ll_{\varepsilon,B} T^{-B}\ .
    \]
    \item If $T^{3-\varepsilon}<|y|$, then
    \[
        |y|^j\frac{\partial^j}{\partial y^j} \Phi_{w_4}^h(y)\ll_{\varepsilon,j} T^{3+2\delta+\varepsilon}(T+|y|^{1/3})^j
    \]
    for any $j\in\IN_0$.
\end{enumerate}
Let $g\in \CH(A_0)$ such that $g$ also satisfies (\ref{eq:rapid_decay_awayfrommu}) and (\ref{eq:BoundSpec_h}), and there is $\epsilon\in\{\pm\}$ such that $g(\mu)\exp(\epsilon h_\infty(\mu))$ satisfies (\ref{eq:bound_derivatives_h}).

\begin{enumerate}[(a)]
\setcounter{enumi}{2}
    \item
    For $0<|y|\leq T^{3-\varepsilon}$ and any $B\geq 0$, we have
    \[
        \Phi_{w_4}^g(y)\ll_{\varepsilon,B} T^{-B}\ .
    \]
    \item If $T^{3-\varepsilon}<|y|$, then
    \[
        |y|^j\frac{\partial^j}{\partial y^j} \Phi_{w_4}^g(y)\ll_{\varepsilon,j} T^{3+2\delta+\varepsilon}(T+|y|^{1/3})^j
    \]
    for any $j\in\IN_0$.
\end{enumerate}
The same holds for $\Phi_{w_5}$.
\end{lemma}
The assumptions in (c) and (d) might seem strange, but they will be exactly satisfied by the test functions coming from the $W$-part of the approximate functional equation.
\begin{proof}
In this proof we do not use $\varepsilon$-convention.

Assuming the statements hold for $w_4$. Because $\spec(\mu)=\spec(-\mu)$, we have
\[
    \Phi_{w_5}^h(y)=\Phi_{w_4}^{\tilde{h}}(-y),
\]
with $\tilde{h}(\mu):=h(-\mu)$. Also (\ref{eq:rapid_decay_awayfrommu}), (\ref{eq:BoundSpec_h}), and (\ref{eq:bound_derivatives_h}) are not affected and thus applying the results for $w_4$ and $-\tilde{\mu}$ gives the results for $w_5$.

The first two points are simply \cite{BB20}*{Lemma 8} and proven in \cite{BB20}*{12. Proof of Lemma 8}.

For part (c) and (d) we follow the same ideas but we have to account for the additional oscillation coming from $h_\infty$. We fix $\epsilon=-1$. The choice does not affect any of the below calculations.

As seen in Lemma \ref{lemma:derivative_hinf} the oscillation in $\mu_i$ or $\nu_i$ of $h_\infty$ is almost linear for $\mu$ close to $\tilde{\mu}$. When applying Lemma \ref{lemma:intbyparts} below the added oscillation will shift the critical range by a constant $C$ that only depends on the ratios of the $\tilde{\mu}_i$ and $\tilde{\nu}_i$ and is bounded from below and above in terms of $c_1$ and $c_2$. The treatment of this critical range afterwards remains the same as in \cite{BB20}*{12. Proof of Lemma 8}.

We start by inserting the integral representation (\ref{eq:lemma4})
into the definition of $\Phi_{w_4}$. We can replace $g(\mu)\exp(-h_\infty(\mu))$ at the cost of a negligible error by a real analytic function $h$ that also satisfies (\ref{eq:rapid_decay_awayfrommu}), (\ref{eq:bound_derivatives_h}), and (\ref{eq:BoundSpec_h}) and has support in $\min_{w\in \mathcal{W}}\|\mu-w(\tilde{\mu})\|\leq T^{\delta+\varepsilon}$. We further only treat $h$ restricted to the ball around $\tilde{\mu}$ below. The treatment of the $w(\tilde{\mu})$ is analogous with $\tilde{\mu}$ replaced by $w(\tilde{\mu})$.

We parametrize the $\mu$-plane by $\mu_3$ and $\nu_1$. We start with the integral in (\ref{eq:lemma4}) involving the $\tilde{K}$-function:
\[
    \int_{\Re\mu=0}h(\mu)\int_{0}^\infty \exp\Big(h_\infty(\mu)-\mu_3\log\frac{\pi^3|y|}{u^{3/2}}\Big)\tilde{K}_{3\nu_1}(2\sqrt{u})\exp\Big(\frac{2i \pi^3y}{u}\Big)\frac{\mathrm{d}u}{u}\spec(\mu)\mathrm{d}\mu\ .
\]
Recall that the $u$-integral is not absolutely convergent at $0$, but as remarked below (\ref{eq:lemma4})
this causes no difficulties.

We want to utilize oscillation in the $\mu_3$-integral. The oscillation is $\exp(iH(\mu_3))$ with
\[
    H(\mu_3)=\Im(\mu_3)\log\Big(\frac{u^{3/2}}{\pi^3|y|}\Big)+\Im(h_\infty(\mu))\ .
\]
We assumed that $\|\mu-\tilde{\mu}\|\leq T^{\delta+\varepsilon}$ and can thus apply Lemma \ref{lemma:derivative_hinf} to get that
\[
\frac{\partial}{\partial\Im(\mu_3)}H(\mu_3)=\log\Big(\frac{u^{3/2}}{\pi^3|y|}\cdot\big|\frac{1}{4}\frac{\tilde{\mu}_3^2}{\tilde{\mu}_{1}\tilde{\mu}_{2}}\big|^{1/2}\Big)+O(T^{-1+\delta+\varepsilon}),\quad \frac{\partial^j}{\partial\Im(\mu_3)^j}H(\mu_3)\ll_j T^{1-j}
\]
for all $j\geq 2$ and thus
\[
    |\frac{\partial}{\partial\Im(\mu_3)}H(\mu_3)|\gg 1
\]
 except if $u\asymp |y|^{2/3}$. If $u\asymp |y|^{2/3}$ is not the case, we can apply Lemma \ref{lemma:intbyparts} and the contribution is negligible. We can thus restrict the $u$ integral to $u\asymp |y|^{2/3}$ at the cost of a negligible error.

Assuming $|y|\leq T^{3-\varepsilon}$ we conclude $u^{1/2}\ll T^{1-\varepsilon/3}$ and we can insert the uniform asymptotic formula (\ref{eq:4.17}) for the $\tilde{K}$-function.
The error term is negligible for $M$ large enough as $|\nu_1|\asymp T$. Now, we can apply Lemma \ref{lemma:intbyparts} to the $\nu_1$ integral with
\[
    H(\nu_1)=\omega(4\sqrt{u},\Im (3\nu_1))+\Im(h_\infty(\mu))\ .
\]
We can again use Lemma \ref{lemma:derivative_hinf} together with $u^{1/2}\ll T^{1-\varepsilon/3}$ and the derivatives of $\mathrm{arccosh}$ to see that
\[
    \frac{\partial}{\partial\Im(\nu_1)}H(\nu_1)=\mathrm{arccosh}\Big(\frac{\Im(3\nu_1)}{4\sqrt{u}}\Big)+\log\Big|\frac{\tilde{\mu}_1}{\tilde{\mu}_2}\Big|+O(T^{-1+\delta+\varepsilon}),\quad \frac{\partial^j}{\partial\Im(\nu_1)^j}H(\mu_3)\ll_j T^{1-j}
\]
for all $j\geq 2$ and thus
\[
    |\frac{\partial}{\partial\Im(\nu_1)}H(\nu_1)|\gg \log T\ .
\]
The other parameters for Lemma \ref{lemma:intbyparts} are $Y=Q=T$ and $U=T^\delta$ and thus the $\nu_1$ integral is negligible. This completes the proof of (c) for the $\tilde{K}$-part.

If $|y|> T^{3-\varepsilon}$, we differentiate $j$ times with respect to $y$ under the integral sign. Keeping in mind that $u$ is restricted to $u\asymp |y|^{2/3}$, each such differentiation produces a factor $T|y|^{-1}+u^{-1}\asymp (T+|y|^{1/3})|y|^{-1}$. A trivial estimate using (\ref{eq:BoundSpec_h}) and (\ref{eq:4.13}) 
completes the proof of (d) for the $\tilde{K}$-part.

The $\tilde{J}$ part of (\ref{eq:lemma4}) goes analogous with (\ref{eq:4.20}) instead of (\ref{eq:4.17}).
\end{proof}

\begin{lemma}
\label{lemma:fine_bound_inttransfrom_long}
Let $h\in \CH(A_0,\tilde{\mu})$ and let $\mathcal{Y}:=\min(|y_1|^{1/3}|y_2|^{1/6},|y_1|^{1/6}|y_2|^{1/3})$.
\begin{enumerate}[(a)]
    \item If $\mathcal{Y}\leq T^{1-\varepsilon}$, then for any $B\geq 0$
    \[
        \Phi_{w_6}^h(y_1,y_2)\ll_{\varepsilon,B} T^{-B}\ .
    \]
    \item If $T^{1-\varepsilon}<\mathcal{Y}$, then
    \[
        |y_1|^{j_1}|y_2|^{j_2}\frac{\partial^{j_1}}{\partial y_1^{j_1}}\frac{\partial^{j_2}}{\partial y_2^{j_2}} \Phi_{w_6}^h(y_1,y_2)\ll_{\varepsilon,j} T^{3+2\delta+\varepsilon}(T+|y_1|^{1/2}+|y_1|^{1/3}|y_2|^{1/6})^{j_1}(T+|y_2|^{1/2}+|y_2|^{1/3}|y_1|^{1/6})^{j_2}
    \]
    for any $j_1,j_2\in\IN_0$.
\end{enumerate}
Let $g\in \CH(A_0)$ such that  $g$ also satisfies (\ref{eq:rapid_decay_awayfrommu}) and (\ref{eq:BoundSpec_h}), and $g(\mu)\exp(- h_\infty(\mu))$ satisfies (\ref{eq:bound_derivatives_h})\ .
\begin{enumerate}[(a)]
\setcounter{enumi}{2}
    \item If $\mathcal{Y}\leq T^{1-\varepsilon}$, then for any $B\geq 0$
    \[
        \Phi_{w_6}^g(y_1,y_2)\ll_{\varepsilon,B} T^{-B}\ .
    \]
\end{enumerate}
\end{lemma}
\begin{proof}
In this proof, we do not use $\varepsilon$-convention.

The first two points are simply \cite{BB20}*{Lemma 9} and are proven in \cite{BB20}*{13. Proof of Lemma 9}.

For part (c), we follow the same ideas as in \cite{BB20}*{13. Proof of Lemma 9}, but we again have to account for the additional oscillation coming from $h_\infty$.

This added oscillation changes the proof in the following way. When applying integration by parts in the non-critical ranges, we often encounter a $\log$  that needs to be bounded from below by a constant, which is the case except when certain variables are of the same size. The added oscillation adds a constant factor inside the logarithm, but this factor is bounded independently of $T$ in terms of $c_1,c_2$, and thus does not change the fact that imbalanced variables lead to a negligible contribution. This is very similar to the above proof of Lemma \ref{lemma:fine_bound_inttransfrom_voronoi}.

We demonstrate this for the case $(\epsilon_1,\epsilon_2)=(+,+)$. The other cases go exactly as in \cite{BB20} with the same modifications.

Without loss of generality, we can assume $y_1\geq y_2>0$. We again replace $g(\mu)\exp(- h_\infty(\mu))$ at the cost of a negligible error with a real analytic function $h$ that also satisfies (\ref{eq:rapid_decay_awayfrommu}), (\ref{eq:bound_derivatives_h}), (\ref{eq:BoundSpec_h}), and is supported in $\min_{w\in W}\|\mu-w(\tilde{\mu})\|\leq T^{\delta+\varepsilon}$. As $K_{w_6}^{++}$ is invariant under the Weyl group, we assume without loss of generality that $h$ is only supported in the positive Weyl chamber $\Im(\nu_1),\Im(\nu_2)>0$ and that $\tilde{\mu}$ lies in this chamber. There we have
\[
    \frac{\cos(\frac{3}{2}\pi\nu_1)\cos(\frac{3}{2}\pi\nu_2)}{\cos(\frac{3}{2}\pi\nu_3)}=\frac{1}{2}+O_B(T^{-B})
\]
for any $B\geq 0$. Using this, we insert (\ref{eq:5.8})
and the integral representation (\ref{eq:5.7}) for $K_{w_6}^{++}$.

We start by considering the $\mu_2$ integral. This is equal to
\[
    \int_{(0)}h(\mu)\Big(\frac{y_1}{y_2}\Big)^{\frac{1}{2}\mu_2}u^{3\mu_2}\exp(h_\infty(\mu))\spec(\mu)\mathrm{d}\mu_2\ .
\]
The oscillation is $\exp(iH(\mu_2))$ with
\[ 
   H(\mu_2)=\Im(\mu_2)\log\Big(\frac{u^{3}y_1^{1/2}}{y_2^{1/2}}\Big)+\Im(h_\infty(\mu))\ .
\]
Using Lemma \ref{lemma:derivative_hinf} we get
\[
\frac{\partial}{\partial\Im(\mu_2)}H(\mu_2)=\log\Big(\frac{u^{3}y_1^{1/2}}{y_2^{1/2}}\cdot\frac{|\tilde{\mu}_2|}{4|\tilde{\mu_3}\tilde{\mu_1}|^{1/2}}\Big)+O(T^{-1+\delta+\varepsilon}),\quad 
\frac{\partial^j}{\partial\Im(\mu_2)^j}H(\mu_2)\ll_j T^{1-j}
\]
for all $j\geq 2$, and thus 
\[
    \Big|\frac{\partial}{\partial\Im(\mu_2)}H(\mu_2)\Big|\gg T^{-\epsilon}
\]
except if 
\begin{equation}
    \label{eq:formula_u++}
    u=\frac{y_2^{1/6}}{y_1^{1/6}}C(1+O(T^{-\epsilon/2})),
\end{equation}
where $C=\left(\frac{|\tilde{\mu}_2|}{4|\tilde{\mu_3}\tilde{\mu_1}|^{1/2}}\right)^{-1/3}$.

If this is not the case, we apply Lemma \ref{lemma:intbyparts} with $R=T^{-\varepsilon},U=T^\delta,Y=Q=T$, and the contribution is negligible. We are left with $u\asymp\frac{y_2^{1/6}}{y_1^{1/6}}$, and thus in particular
\begin{equation}
    \label{eq:formula_ys++}
    y^{1/2}_1\sqrt{1+u^2}\asymp y_1^{1/2},\quad y_2^{1/2}\sqrt{1+u^{-2}}\asymp y_2^{1/3}y_1^{1/6}\ .
\end{equation}
Using this together with $\mathcal{Y}\leq T^{1-\varepsilon}$ allows us to insert (\ref{eq:4.17})
into the second Bessel function in (\ref{eq:5.7}) with $t=\Im (3\nu_3)$, $x=4\pi y_2^{1/2}\sqrt{1+u^{-2}}\ll |t|^{1-\varepsilon}$, and $M$ large enough so that the error is negligible. For the first Bessel function, we distinguish two cases:

\textbf{Case 1: $y_2\geq y_1T^{-\varepsilon}$.} Using $\mathcal{Y}\leq T^{1-\varepsilon}$, we have
\begin{equation}
    \label{eq:bound_ys}
    y_1^{1/2}\leq y_2^{1/2}T^{\frac{\varepsilon}{2}}\leq y_2^{1/3}y_1^{1/6}T^{\varepsilon/2}\leq T^{1-\varepsilon/2}\ .
\end{equation}
We insert (\ref{eq:4.17})
for the first Bessel function in (\ref{eq:5.7}) at the cost of a negligible error. Looking at the $\nu_3$-integral, we have
\begin{align*}
    \int_{(0)}\exp\Big(\pm i\omega(4\pi y_2^{1/2}\sqrt{1+u^{-2}},|3\nu_3|)\pm i\omega(4\pi y_1^{1/2}\sqrt{1+u^{2}},|3\nu_3|)-h_\infty(\mu)\Big)\\
    \cdot h(\mu)f^\pm_M(4\pi y_2^{1/2}\sqrt{1+u^{-2}},|3\nu_3|)f^\pm_M(4\pi y_1^{1/2}\sqrt{1+u^{2}},|3\nu_3|)\spec(\mu)\mathrm{d}\nu_3\ ,
\end{align*}
where $f_M^+=f_M$ and $f_M^-=\overline{f}_M$ are as in (\ref{eq:4.17}).

With the phase function $H(\nu_3)=\pm \omega(4\pi y_2^{1/2}\sqrt{1+u^{-2}},|3\nu_3|)\pm \omega(4\pi y_1^{1/2}\sqrt{1+u^{2}},|3\nu_3|)-\Im(h_\infty(\mu))$, we have
\begin{equation*}
    \frac{\partial}{\partial\Im(\nu_3)}H(\nu_3)=\pm 3\mathrm{arccosh}\Big(\frac{|3\nu_3|}{4\pi y_2^{1/2}\sqrt{1+u^{-2}}}\Big)\pm 3\mathrm{arccosh}\Big(\frac{|3\nu_3|}{4\pi y_1^{1/2}\sqrt{1+u^{2}}}\Big)+3\log\Big|\frac{\tilde{\mu}_3}{\tilde{\mu}_1}\Big|^{1/2}+O(T^{-1+\delta+\varepsilon})
\end{equation*}
and $\frac{\partial^j}{\partial\Im(\nu_3)^j}H(3\nu_3)\ll T^{1-j}$ for $j\geq 2$. Using $\mathrm{arccosh}(x)=\log(2x)+O(x^{-2})$ together with (\ref{eq:formula_ys++}) and (\ref{eq:bound_ys}), we can bound the absolute value of the first derivative from below by $T^{-\varepsilon}$ except if the two $\pm$ signs are different and
\[
    \frac{y_1^{1/2}\sqrt{1+u^{2}}}{y_2^{1/2}\sqrt{1+u^{-2}}}=\left|\frac{\tilde{\mu}_3}{\tilde{\mu}_1}\right|^{\pm1/2} +O(T^{-\varepsilon/2})\ .
\]
In all other cases, we can use Lemma \ref{lemma:intbyparts} with $R=T^{-\varepsilon},U=T^\delta,Y=Q=T$ to prove that the contribution is negligible. In the above case, we can use (\ref{eq:formula_u++}) to localize $|u-C|\ll T^{-\varepsilon/2}$ at the  cost of a negligible error for some nonzero constant $C$, depending only on the relative sizes of the $\tilde{\mu}_i$ and $\tilde{\nu}_i$, and thus being bounded from below and from above in terms of $c_1,c_2$.

We insert a smooth, compactly supported cutoff function $\psi(u)$ with $\psi^{(j)}(u)\ll_j T^{\frac{\varepsilon}{2}j}$ for $j\geq 0$ into the integral and consider the $u$ integral
\begin{align*}
    \int_{0}^\infty \exp\Big(\epsilon i \omega(4\pi y_2^{1/2}\sqrt{1+u^{-2}},|3\nu_3|)-\epsilon i \omega(4\pi y_1^{1/2}\sqrt{1+u^{2}},|3\nu_3|) \Big)u^{3\mu_2}\\
    \cdot \psi(u)f^\epsilon_M(4\pi y_2^{1/2}\sqrt{1+u^{-2}},|3\nu_3|)f^{-\epsilon}_M(4\pi y_1^{1/2}\sqrt{1+u^{2}},|3\nu_3|)\frac{\mathrm{d}u}{u}
\end{align*}
for $\epsilon\in \{\pm\}$. With the phase function
\[
    H(u)=\epsilon \omega(4\pi y_2^{1/2}\sqrt{1+u^{-2}},|3\nu_3|)-\epsilon \omega(4\pi y_1^{1/2}\sqrt{1+u^{2}},|3\nu_3|)+\Im( 3\mu_2)\log u
\]
we have $|H^\prime(u)|\asymp T$ and $H^{(j)}(u)\ll T$ for $j\geq 2$. Using Lemma \ref{lemma:intbyparts} with $R=Y=T,U=T^{-\varepsilon/2}, Q=1$, we see that the $u$-integral is negligible.

\textbf{Case 2: $y_2\leq y_1T^{-\varepsilon}$.} We use (\ref{eq:4.12})
with $t=\Im (3\nu_3),x=4\pi y_1^{1/2}\sqrt{1+u^2}$ for the first Bessel function in (\ref{eq:5.7}). We consider the $\nu_3$ integral
\begin{align*}
    \int_{(0)}\exp\Big(\pm i\omega(4\pi y_2^{1/2}\sqrt{1+u^{-2}},|3\nu_3|)+i|3\nu_3|v+h_\infty(\mu)\Big)\\
    \cdot h(\mu)f^\pm_M(4\pi y_2^{1/2}\sqrt{1+u^{-2}},|3\nu_3|)\spec(\mu)\mathrm{d}\nu_3\ .\nonumber
\end{align*}
With the phase function $H(\nu_3)=3v\pm \omega(4\pi y_2^{1/2}\sqrt{1+u^{-2}},|3\nu_3|)+\Im (h_\infty(\mu))$ using Lemma \ref{lemma:derivative_hinf}, we have 
\[
    \frac{\partial}{\partial\Im(\nu_3)}H(\nu_3)=3v\pm 3\mathrm{arccosh}\Big(\frac{|3\nu_3|}{4\pi y_2^{1/2}\sqrt{1+u^{-2}}}\Big)+3\log\Big(\frac{|\tilde{\mu}_3|^{1/2}}{|\tilde{\mu}_1|^{1/2}}\Big)+O(T^{-1+\delta+\varepsilon})
\]
and $ \frac{\partial^j}{\partial\Im(\nu_3)^j}H(\nu_3)\ll T^{1-j}$ for $j\geq 2$ if $|\nu_3|\asymp T$. Lemma \ref{lemma:intbyparts} implies that the $\nu_3$ integral is negligible except if
\[
    \cosh v \asymp \frac{T}{y_2^{1/2}y_1^{1/6}}\geq T^\varepsilon\ .
\]
As above, we insert a smooth cutoff function $\psi(v)$ with $\psi^{(j)}(v)\ll_j 1$ for $j\geq 0$ at the cost of a negligible error, and the $v$ integral becomes
\[
    \int_\IR \psi(v)\cos(4\pi y_1^{1/2}\sqrt{1+u^2}\sinh v)\exp(i|3\nu_3|v)\mathrm{d}v\ .
\]
With the phase function $H(v)=|3\nu_3|v\pm 4\pi y_1^{1/2}\sqrt{1+u^2}\sinh v$ we have
$|H^\prime(v)|\asymp T\Big(\frac{y_1}{y_2}\Big)^{1/3}$ and $H^{(j)}(v)\ll T\Big(\frac{y_1}{y_2}\Big)^{1/3}$ for $j\geq 2$. Lemma \ref{lemma:intbyparts} with $Y=R=T\Big(\frac{y_1}{y_2}\Big)^{1/3}, Q=1,U=1$ implies that the $v$ integral is negligible. This concludes the $(++)$ case.

The remaining cases $(--),(\pm,\mp)$ are analogous to \cite{BB20}*{Proof of Lemma 9} using the other integral representation and results regarding Bessel functions from \cite{BB20}*{Sections 4 and 5}.
\end{proof}

\begin{lemma}
\label{lemma:bound_fourier_voronoi}
Let $s\in \IC$ with $|\Im(s)|\leq T^\varepsilon$ and $\Re(s)$ be bounded independently of $T$. Let $h\in \CH(A_0)$ satisfy (\ref{eq:rapid_decay_awayfrommu}) and (\ref{eq:BoundSpec_h}). Let $w$ be a smooth function with compact support in $\IR^+$ and bounded derivatives.
\begin{enumerate}[(a)]
    \item If $T^{3-\varepsilon}\leq |\Xi|\leq T^{3+\varepsilon}$, then
    \[
        \int_\IR y^{-s}w(y)\Phi_{w_4}^h(y^2\Xi)\mathrm{d}y\ll_\varepsilon T^{\frac{3}{2}+2\delta+\varepsilon}
    \]
    independent of $\Xi$.
    \item Else,
    \[
        \int_\IR y^{-s}w(y)\Phi_{w_4}^h(y^2\Xi)\mathrm{d}y\ll_{B,\varepsilon} T^{-B}
    \]
    for any $B\geq 0$.
\end{enumerate}
The same holds for $w_5$ instead of $w_4$.
\end{lemma}
\begin{proof}
    We can replace $w_4$ with $w_5$ by the same argument as in the proof of Lemma \ref{lemma:fine_bound_inttransfrom_voronoi}.
    Plugging in the definition of $\Phi_{w_4}$, we have
    \[
        \int_\IR y^{-s}w(y)\Phi_{w_4}^h(y^2\Xi)\mathrm{d}y=\int_{\Re\mu=0}h(\mu)\spec(\mu)\int_\IR y^{-s}w(y)K_{w_4}(y^2\Xi,\mu)\mathrm{d}y\mathrm{d}\mu\ .
    \]
    
    As $h$ satisfies (\ref{eq:rapid_decay_awayfrommu}), we can start by truncating the $\mu$ integral to $\min_{w\in W}\|\mu-w(\tilde    {\mu})\|\leq T^{\delta+\varepsilon}$ at the cost of a negligible error.
    
    We then insert the definition of $K_{w_4}$. The inner integral reads
    \begin{align*}
        \int_\IR y^{-s}w(y)\int_{(0)}|\Xi|^{-s_1}|y|^{-2s_1}\tilde{G}^\epsilon(s_1,\mu)\frac{\mathrm{d}s_1}{2\pi i}\mathrm{d}y\ .
    \end{align*}
    Let $\hat{w}$ be the Mellin transform of $w$. We have
    \begin{align*}
        \int_\IR y^{-s}w(y)\int_{(0)}|\Xi|^{-s_1}|y|^{-2s_1}\tilde{G}^\epsilon(s_1,\mu)\frac{\mathrm{d}s_1}{2\pi i}\mathrm{d}y
        =\int_{(0)}\hat{w}(1-s-2s_1)|\Xi|^{-s_1}\tilde{G}^\epsilon(s_1,\mu)\frac{\mathrm{d}s_1}{2\pi i}\ .
    \end{align*}
    We have $\tilde{G}^\epsilon(s_1,\mu)\ll T^{3\Re s_1-\frac{3}{2}}$ away from poles. If $|\Xi|>T^{3+\varepsilon}$, we shift the $s_1$ contour to $B>0$ large. All poles are negligible due to the rapid decay of $\hat{w}$, and we can bound the integrand by $T^{-\frac{3}{2}-B\varepsilon}$. Because of the rapid decay of $\hat{w}$, we can truncate the $s_1$ integral at $|\Im (s_1)|\leq T^\varepsilon$. If $|\Xi|<T^{3-\varepsilon}$, we do the same but shift to the left.
    
    If $T^{3-\varepsilon}\leq |\Xi|\leq T^{3+\varepsilon}$, we truncate the $s_1$ contour as above and estimate everything trivially using that $h$ satisfies (\ref{eq:BoundSpec_h}). 
\end{proof}

The three previous lemmata all accounted for the test functions from both the $V$-part and the $W$-part of the approximate functional equation. The following Lemma will only account for the $V$-part. It is not needed for the $W$-part, as the asymmetry, together with Lemma \ref{lemma:fine_bound_inttransfrom_long}, will be enough to bound this.

\begin{lemma}
\label{lemma:bound_fourier_long}
Let $s\in\IC$ with $|\Im(s)|\leq T^\varepsilon$ and $\Re(s)$ be bounded independently of $T$. Let $h\in \CH(A_0,\tilde{\mu})$ and let $w$ be a smooth function with compact support in $\IR^+$ and bounded derivatives.

Then the following holds for $|\Xi_2|\ll T^{2+\frac{1}{120}}$:
\begin{enumerate}[(a)]
    \item
    \begin{align*}
        \int_\IR y^{-s}w(y)\Phi_{w_6}(y^2\Xi_1,\Xi_2)\mathrm{d}y\ll_\varepsilon T^{\frac{3}{2}-\frac{1}{6}+\frac{1}{360}+2\delta+\varepsilon}
    \end{align*}
    uniform in $\Xi_2$ with polynomial dependence on $s$.
    \item For $|U_1|\geq T^{1-\frac{1}{120}}$ and $T^{3-\frac{1}{120}}\leq |\Xi_2\Xi_1^{-1}U_1^3|\leq T^{3+\frac{1}{120}}$, we have
    \begin{align*}
        \int_\IR y^{-s}w(y)\Phi_{w_6}(y^2\Xi_1,\Xi_2)e(y U_1)\mathrm{d}y\ll_\varepsilon T^{\frac{3}{2}-\frac{1}{240}+2\delta+\varepsilon}\ 
    \end{align*}
    uniform in $\Xi_1,\Xi_2$ with a polynomial dependence on $s$.
\end{enumerate}
\end{lemma}
\begin{rem}
The constant $\frac{1}{240}$ is not optimized and could be improved at the cost of a longer proof.
\end{rem}
\begin{proof}
For this proof, we want to look at each of the balls around $w(\tilde{\mu})$ individually. To do this, we first truncate the $\mu$ integral to $\min_{w\in\mathcal{W}}\|\mu-w(\tilde{\mu})\|\leq T^{\delta+\epsilon}$ at the cost of a negligible error and split this into $6$ balls around the $w(\tilde{\mu})$ for $w\in\mathcal{W}$. Even if the contours are now truncated, we can still shift them because the additional terms are negligible due to the rapid decay of $h$ away from $w(\tilde{\mu})$. For the rest of the proof, we only treat the ball around $\tilde{\mu}$, which can be replaced by $w(\tilde{\mu})$ everywhere to analyze the other balls.

\begin{enumerate}[(a)]
    \item Plugging in the definition of $\Phi_{w_6}$ we have
    \[
    \int_\IR y^{-s}w(y)\Phi_{w_6}(y^2\Xi_1,\Xi_2)\mathrm{d}y
        =\ \int_{\Re(\mu)=0}h(\mu)\mathrm{spec}(\mu)\int_\IR y^{-s}w(y)K_{w_6}^{\epsilon_1,\epsilon_2}(y^2\Xi_1,\Xi_2,\mu)\mathrm{d}y\mathrm{d}\mu\ .
    \]

    Define $\hat{w}(s)$ as the Mellin transform of $w$. After plugging in the definition of $K_{w_6}$ we get that the inner integral is equal to
    \begin{align*}
        &\int_\IR y^{-s}w(y)K_{w_6}^{\epsilon_1,\epsilon_2}(y^2\Xi_1,\Xi_2,\mu)\mathrm{d}y\\
        =&\int_{(0)}\int_{(0)}G(s,\mu)S^{\epsilon_1,\epsilon_2}(s,\mu)|4\pi^2\Xi_1|^{-s_1}|4\pi^2\Xi_2|^{-s_2}\hat{w}(1-2s_1-s)\frac{\mathrm{d}s_1\mathrm{d}s_2}{(2\pi i)^2}\ .
    \end{align*}
    We start by truncating the $s_1$ contour at $|s_1|\leq T^\varepsilon$ by the rapid decay of $\hat{w}$. We also shift the $s_2$ contour to real part $\varepsilon$ for $|s_2|\leq T^B$ and discard the rest making a negligible error for a large enough $B$.
    
    We then partition the $s_2$ integral into dyadic ranges $U_2$ such that  $t_2:=\Im(s_2)\asymp U_2$. Furthermore, we define $\tilde{r}_i:=\Im(\tilde{\mu}_i)$, $B_2:=\min_{i} |t_2+\tilde{r}_i|$ and $1\leq l\leq 3$ the index that realizes this minimum. We also sum over $B_2$ in dyadic ranges. We distinguish the following cases:
    
    \textbf{Case 1: $|U_2|\leq T^{1/2}$\ .} In this case, we estimate the integrant trivially by $T^{-3+\frac{1}{4}+\varepsilon}$ and the $t_2$-integral and the $\mu$ integral bring this to $T^{\frac{3}{4}+2\delta+\varepsilon}$.
    
    \textbf{Case 2: $B_2\leq T^{1-\varepsilon}$\ .} We shift the $\mu_l$ contour to $B>0$ and the $\mu_j$ contour to $-B$ for some $j\neq l$ without picking up any residues as we have truncated the $\mu$ integral at the start. We then estimate trivially by $T^{O(1)-B}B_2^B\ll T^{O(1)-\varepsilon B}$, which is negligible for large enough $B$.
    
    \textbf{Case 3: $|U_2|> T^{1/2}$ and $B_2> T^{1-\varepsilon}$\ .}
    Define $r_i:=\Im(\mu_i)$. As $|\mu_i-\tilde{\mu}_i|\leq T^{\delta+\varepsilon}$, we have $B_2\asymp |t_2+r_l|$. We also truncate the $t_2$-integral at some $T^{B}$ with $B$ large enough. As $B_2> T^{1-\varepsilon}$ and $U_2> T^{\frac{1}{2}}$, we are away from all poles of the Gamma functions and can apply Stirling's formula (\ref{eq:stirling}) resulting in:
    \begin{equation}
    \label{eq:stirling_t2}
        |\Xi_1\Xi_2|^\varepsilon\int_{|t_1|\ll T^\varepsilon}\int_{\IR}\exp\Big(-\frac{\pi}{2}h^{\epsilon_1,\epsilon_2}(r,t)+iH(t,r)\Big)F(t,r)
    \end{equation}
    with $h^{\epsilon_1,\epsilon_2}$ as defined in (\ref{eq:def_expbehaviour_w6}),
    \[
        H(t,r):= \sum_{j=1}^3 t_2(\log(|t_2+r_j|)-1)-t_2(\log(|t_2+t_1|)-1)-t_2\log(|\Xi_2|)
    \]
    and $F$ is a smooth function with support in $t_2\asymp U_2$ and $B_2\asymp |t_2+r_l|\leq  |t_2+r_j|+T^{\delta+\varepsilon}$ for $j\neq l$ that satisfies
    \[
        \frac{\partial^n}{\partial t_2^n}F(t,r)\ll T^{-3/2+\varepsilon}(T+|U_2|)^{-1}|U_2|^{1/2}B_2^{-1/2}E_2^{-n}
    \]
    for all $n\geq 0$, where $E_2:=\min(|U_2|,B_2)$.
    
    As $F$ is only supported away from the kink points of $h^{\epsilon_1,\epsilon_2}$ we either have $h^{\epsilon_1,\epsilon_2}(r,t)=0$ or $h^{\epsilon_1,\epsilon_2}(r,t)\geq T^\varepsilon$ making the whole case negligible. We can thus assume from now on that $h^{\epsilon_1,\epsilon_2}(r,t)=0$. To apply Lemma $\ref{lemma:intbyparts}$ we need to understand the derivatives of $H$. We have
    \[
        \frac{\partial^n}{\partial t_2^n}H(t,r)\ll E_2^{1-n}
    \]
    for $n\geq 2$ and 
    \begin{equation*}
        \frac{\partial}{\partial t_2}H(t,r)=\log\Big(\frac{\prod_i |t_2+r_i|}{|t_2+t_1||\Xi_2|}\Big)\ .
    \end{equation*}
    
    We now further sum over the size of the first derivative in dyadic ranges $B_5\asymp| \frac{\partial}{\partial t_2}H(t,r)|$. If $B_5\gg E_2^{-1/2}T^\varepsilon$, we can apply Lemma \ref{lemma:intbyparts} with $Y=Q=U=E_2$ and the contribution is negligible.
    
    We are left with the cases $B_5\ll E_2^{-1/2}T^\varepsilon$. Let $A_l:=A_l(U_2,B_2,B_5)$ be the region where $t_2\asymp U_2$, $B_2\asymp |t_2+r_l|\leq_\varepsilon |t_2+r_j|$, and 
    \begin{equation*}
        \left|\log\left|\frac{\prod_i (t_2+r_i)}{(t_2+t_1)\Xi_2}\right|\right|\ll  E_2^{-1/2}T^\varepsilon.
    \end{equation*}
    This is equivalent to
    \begin{equation}
        \label{eq:bound_derivate1}
        \left|\frac{\prod_i (t_2+r_i)}{(t_2+t_1)\Xi_2}-\alpha_2\right|\ll  E_2^{-1/2}T^\varepsilon
    \end{equation}
    for some $\alpha_2\in\{\pm1\}$. Let $A_l^{\alpha_2}$ be the region corresponding to a specific choice for $\alpha_2$. We want to bound $\mathrm{meas}(A_l^{\alpha_2})$. We can transform (\ref{eq:bound_derivate1}) into
    \begin{align*}
        |\prod_i (t_2+r_i) - (t_2+t_1)\alpha_2\Xi_2|&\ll |\Xi_2U_2E_2^{-1/2}T^\varepsilon|\\
        \Leftrightarrow \prod_{i=1}^3 |t_i-q_i|&\ll |\Xi_2U_2E_2^{-1/2}T^\varepsilon|\ ,
    \end{align*}
    where the $q_i$ are complex numbers independent of $t_2$. Then
    \[
        \min_i|t_2-q_i|\ll |\Xi_2U_2E_2^{-1/2}T^\varepsilon|^{1/3}
    \]
    and thus $\mathrm{meas}(A_l^{\alpha_2})\ll |\Xi_2U_2E_2^{-1/2}T^\varepsilon|^{1/3}$.
    
    Coming back to (\ref{eq:stirling_t2}) and estimating the integrant trivially we have
    \begin{align*}
        &|\Xi_1\Xi_2|^\varepsilon\int_{|t_1|\ll T^\varepsilon}\int_{\IR}\exp\Big(-\frac{\pi}{2}h^{\epsilon_1,\epsilon_2}(r,t)+iH(t,r)\Big)F(t,r)\\ \ll &\max_{|U_2|>T^{1/2},B_2>T^{1-\varepsilon}} T^{-3/2+\varepsilon}(T+|U_2|)^{-1}|U_2|^{1/2}B_2^{-1/2}|\Xi_2U_2E_2^{-1/2}T^\varepsilon|^{1/3}\ .
    \end{align*}
    The maximum is attained for $|U_2|\asymp T$, $B_2\asymp T^{1-\varepsilon}$ and is 
    \[
        T^{-\frac{5}{2}+\frac{1}{6}+\varepsilon}|\Xi_2|^{\frac{1}{3}}\ .
    \]
    Together with the $\mu$ integral and the bound on $|\Xi_2|$ we have the total bound $T^{\frac{3}{2}-\frac{1}{6}+\frac{1}{360}+2\delta+\varepsilon}$, which was claimed.\qed
    
    \item Plugging in the definition of $\Phi_{w_6}$ we have
    \[
    \int_\IR y^{-s}w(y)\Phi_{w_6}(y^2\Xi_1,\Xi_2)e(y U_1)\mathrm{d}y
        =\ \int_{\Re(\mu)=0}h(\mu)\mathrm{spec}(\mu)\int_\IR y^{-s}w(y)K_{w_6}^{\epsilon_1,\epsilon_2}(y^2\Xi_1,\Xi_2,\mu)e(y U_1)\mathrm{d}y\mathrm{d}\mu\ .
    \]
    Start by looking at the inner integral
    \begin{align*}
        &\int_\IR y^{-s}w(y)K_{w_6}^{\epsilon_1,\epsilon_2}(y^2\Xi_1,\Xi_2,\mu)e(y U_1)\mathrm{d}y\ .
    \end{align*}
    Plugging in the definition of the integral kernel, truncating the $s_1$ and $s_2$ integrals at $T^B$ and shifting them to real part $\varepsilon$ leads to
    \begin{align*}
        &\int_\IR\int_{|t_1|\leq T^B}\int_{|t_2|\leq T^B}G(\varepsilon+it,\mu)S^{\epsilon_1,\epsilon_2}(\varepsilon+it,\mu)|4\pi^2\Xi_1|^{-\varepsilon-it_1}|4\pi^2\Xi_2|^{-\varepsilon-it_2}\\&\cdot y^{-2\varepsilon-\Re(s)-2it_1-i\Im(s)}w(y_1)e(y_1U_1)\frac{\mathrm{d}t_1\mathrm{d}t_2}{(2\pi i)^2}\mathrm{d}y_1\ .
    \end{align*}
    We apply Lemma \ref{lemma:intbyparts_asymptotic} with $y^{-2\varepsilon-\Re(s)}w(y_1)$ as the fixed c.s.\ smooth function, $t=2t_1+\Im(s)$ and $x=U_1$. Even though the smooth c.s.\ function depends on $s$, all resulting dependencies will be polynomial in $s$. Discarding the negligible error terms leaves us with
    \begin{align*}
        |U_1|^{1/2}\int_\IR\int_{|t_1|\leq T^B}\int_{|t_2|\leq T^B}G(\varepsilon+it,\mu)S^{\epsilon_1,\epsilon_2}(\varepsilon+it,\mu)|4\pi^2\Xi_1|^{-\varepsilon-it_1}|4\pi^2\Xi_2|^{-\varepsilon-it_2}\\
        \cdot \Big|\frac{2t_1+\Im(s)}{2\pi e U_1}\Big|^{-i(2t_1+\Im(s))}\tilde{w}_{U_1}\Big(\frac{2t_1+\Im(s)}{U_1}\Big)\frac{\mathrm{d}t_1\mathrm{d}t_2}{(2\pi i)^2}\ .
    \end{align*}
    We further replace the $t_2$ contour by
    \[
        \{x\pm iT^{B}:-C\leq x\leq \varepsilon\}\cup\{-C+ix:T^B\geq|x|\geq T^{1+b}\}\cup\{x\pm iT^{1+b}:-C\leq x\leq \varepsilon\}\cup\{\varepsilon+ix:|x|\leq T^{1+b}\}\ ,
    \]
    for some large $C>0$ and $b=\frac{1}{60}$. The first part is negligible if we choose $B$ big enough. To bound the second part we distinguish two cases:
    
    \textbf{Case 1: $|U_1|\geq T^{1+b}$.} We shift the $t_1$ contour to $C$ and estimate the integrant trivially by
    \[
        |\Xi_1^{-1}\Xi_2|^C|U_1t_2^{-1}|^{3C}T^{O(1)}\leq T^{-\frac{1}{2}Cb+O(1)}(T^2t_2^{-2})^C\ ,
    \]
    where we used $|\Xi_2\Xi_1^{-1}U_1^3|\leq T^{3+b/2}$. As $|t_2|\geq T^{1+b}$, the integral is negligible for large enough $C$.
    
    \textbf{Case 2: $|U_1|\leq T^{1+b}$.} The integrant can be bounded by
     \[
        |t_2|^{-2C}|\Xi_2|^CT^{O(1)}
     \]
     and the integral is negligible because $|\Xi_2|\leq T^{2+b/2}<T^{2+2b}\leq |t_2|^2$.
     
     To bound the contribution of $\{x\pm iT^{1+b}:-C\leq x\leq \varepsilon\}$ we also distinguish two cases:
     
     \textbf{Case 1: $|U_1|\geq T^{1+b}$.} We shift the $t_1$ contour to $x$ and estimate everything trivially by
     \[
        |U_1|^{-1/2}T^{-\frac{3}{2}(1+b)}\int_{-C}^\varepsilon (T^{3+3b}\Xi_1\Xi_2^{-1}U_1^{-3})^x\mathrm{d}x\leq T^{-2}\ ,
     \]
     where we used $|\Xi_2\Xi_1^{-1}U_1^3|\leq T^{3+b/2}$.
     Combining this with the $T^{3+2\delta+\varepsilon}$ from the $\mu$ integral results in a bound lower than the claimed one.
     
     \textbf{Case 2: $|U_1|\leq T^{1+b}$.} We again just estimate trivially and use $|\Xi_2|\leq T^{2+b/2}$ to bound everything including the $\mu$ integral by $T^{1}$.
     
     The remaining $t_2$ contour $\{\varepsilon+ix:|x|\leq T^{1+b}\}$ is split into dyadic ranges $|U_2|\leq T^{1+b}$. We further restrict to $t_1\asymp U_1$ making a negligible error by the rapid decay of $\tilde{w}_{U_1}$.
     
     We want to utilize the oscillation of the $t_1$ and the $t_2$ integral. Because this heavily depends on the sizes of the arguments of the Gamma function, we define $\tilde{r}_i:=\Im(\tilde{\mu}_i)$ and let $j$ and $k$ be such that 
     \[
        |t_1-\tilde{r}_j|=\min_i |t_1-\tilde{r}_i|\textnormal{ and }|t_2+\tilde{r}_k|=\min_i |t_2+\tilde{r}_i|\ .
     \]
     We now also sum over $j$ and $k$ and in dyadic ranges over $B_1\asymp |t_1-\tilde{r}_j|$, $B_2\asymp |t_2+\tilde{r}_k|$, and $B_3\asymp|t_1+t_2|$.
     
     In the generic case, all these variables will be roughly of size $T$. Before dealing with this case, we focus on the $B_i$ being small.
     
     \textbf{Some edge cases:} Assume $j\neq k$ and $B_1\leq T^{1-\varepsilon}$. In this case, we shift the $\mu_j$-contour to $C>0$ and $\mu_i$ to $-C$ for $i\neq j,k$ not picking up any poles as we have truncated the $\mu$ integral. We estimate the whole integral by
     \[
        B_1^{C}T^{-C}T^{O(1)},
     \]
     which is negligible for large enough $C$. For the same reason the contribution of $B_2\leq T^{1-\varepsilon}$ is negligible.
     
     If $j=k$ and $B_1\leq T^{1-\varepsilon}$ and also $B_2\geq B_1T^{\varepsilon}$, we do the same thing and estimate everything by
     \[
        \Big(\frac{B_1}{B_2}\Big)^{C}T^{O(1)},
     \]
     which is negligible. Again, the same holds for $B_1\geq B_2T^{\varepsilon}\leq T$.
     
     If $B_3\leq T^{2/3-\beta+\varepsilon}$, we just estimate trivially by
     \[
        (B_1B_2B_3^{-1})^{-1/2}T^{-2}B_1B_3|U_1|^{-1/2}\leq T^{-\frac{3}{2}(1+\beta)+\varepsilon}\ ,
     \]
     where we used that $|U_1|\ll B_3+|U_2|$.
     For $\beta=b$ this achieves a better bound than the claimed one after including the $\mu$ integral.
     
     Because $j=k$ implies $B_3\ll B_1+B_2$ and $j\neq k$ implies $B_3\gg T$, we are left with the following cases
     \begin{itemize}
         \item $j\neq k$, $B_1,B_2\geq T^{1-\varepsilon}$, and $B_3\gg T$.
         \item $j= k$, $B_1,B_2\geq B_3T^{-\varepsilon}$, $B_1\asymp_\varepsilon B_2$, and $B_3\gg T^{2/3-b}$.
         \item $j= k$, $B_1,B_2\geq T^{1-\varepsilon}$, and $B_3\gg T^{2/3-b}$.
     \end{itemize}
     
     Assume that we are in one of these remaining cases. We define $r_i:=\Im(\mu_i)$. As $B_i\geq T^{1/2}$, we have $|t_1-r_j|\asymp B_1$ and $|t_2+r_k|\asymp B_2$. 
     
     The arguments of the Gamma functions are at least $T^{1/2}$ away from poles and thus we can apply Stirling (\ref{eq:stirling}) up to a negligible error and get
     
     \[
        |U_1|^{-1/2}|\Xi_1\Xi_2|^{-\varepsilon}\int_{\IR^2}\exp{\Big(-\frac{\pi}{2}h^{\epsilon_1,\epsilon_2}(r,t)+ig(t,r)\Big)}F(t,r)\mathrm{d}t\ ,
     \]
     where $h^{\epsilon_1,\epsilon_2}(r,t)$ was defined in (\ref{eq:def_expbehaviour_w6}),
     \begin{align*}
         g(t,r)=& -(t_1+t_2)\log \frac{|t_1+t_2|}{e}\\&+ \sum_i (t_1-r_i)\log \frac{|t_1-r_i|}{e}+(t_2+r_i)\log \frac{|t_2+r_i|}{e}\\
         &-t_1\log |4\pi^2\Xi_1|-t_2\log |4\pi^2\Xi_2|\\
         &-(2t_1+\Im(s))\log \frac{|2t_1+\Im(s)|}{2\pi e |U_1|}
     \end{align*}
     and $F$ is a smooth function with support in $t_i\asymp U_i$ that satisfies
     \[
        \frac{\partial^n}{\partial t_1^n}\frac{\partial^m}{\partial t_2^m}F(t,r)\ll_{m,n}T^\varepsilon\frac{1}{(T+|U_1|)(T+|U_2|)}\Big(\frac{B_3}{B_1B_2}\Big)^{1/2}\frac{1}{E_1^nE_2^m},
     \]
     for $m,n\in\IN$, where
     \[
        E_i:=\min(B_i,B_3,|U_i|)\ .
     \]
     Recall that the exponential growth $h^{\epsilon_1,\epsilon_2}(t,r)$ was a piecewise linear non negative function. As we are over $T^{1/2}$ away from its kink points, it is either $0$ or the integral is negligible. We can thus assume that $h(t,r)=0$ and it is enough to analyze
     \begin{equation*}
         J:=J(B_1,B_2,B_3,U_1,U_2):=\int_{\IR^2}\exp{(ig(t,r))}F(t,r)\mathrm{d}t
     \end{equation*}
     for all the choices of $B_i$ and $U_j$. To apply Lemma \ref{lemma:intbyparts} we need to calculate and bound the derivatives of $g$. We have
     \begin{align*}
         g_1(t):=&\frac{\partial}{\partial t_1}g(t,r)=\log \Big |\frac{\prod_i (t_1-r_i)}{(t_1+t_2)Y_1(2t_1+\Im(s))^2}\Big |\\
         g_2(t):=&\frac{\partial}{\partial t_2}g(t,r)=\log \Big |\frac{\prod_i (t_2+r_i)}{(t_1+t_2)Y_2}\Big |\\
         \frac{\partial^m\partial^n}{\partial t_j^m\partial t_i^n}g_i&\asymp B_3^{-m-n}\textnormal{ for } m\geq 1, n\geq 0, i\neq j\\
         \frac{\partial^n}{\partial t_i^n}g_i&\ll E_i^{-n}\textnormal{ for }n\geq 0\ ,
     \end{align*}
     where $Y_1:=\frac{\Xi_1}{U_1^2}$ and $Y_2:=4\pi^2\Xi_2$.
     
     We now localize the size of $|g_1|$ respectively $|g_2|$ as $B_4$ respectively $B_5$. Applying Lemma \ref{lemma:intbyparts} with $Y=Q=E_i$, $U=E_i\min(1, B_{3+i})$ and $R=B_{3+i}$ implies that the integral is negligible except if
     \[
        B_{3+i}\ll T^\varepsilon E_i^{-1/2}
     \]
     for $i\in \{1,2\}$.
     
     If this is the case, define $A_{j,k}:=A_{j,k}(U_1,U_2,B_1,B_2,B_3,r_1,r_2,r_3)$ to be the set of $(t_1,t_2)$ such that  
     \[
        |g_i(t)|\ll T^\varepsilon E_i^{-1/2}, t_i\asymp U_i, |t_1-r_j|\asymp B_1\leq |t_1-r_l|,|t_2+r_k|\asymp B_2\leq |t_2+r_l|\textnormal{, and} |t_1+t_2|\asymp B_3\ .
     \]
     We can bound $J$ by
     \begin{equation}
         \label{eq:bound_J_general}
         T^\varepsilon\frac{1}{(T+|U_1|)(T+|U_2|)}\Big(\frac{B_3}{B_1B_2}\Big)^{1/2}\mathrm{meas}(A_{j,k})+O(T_B^{-B})\ .
     \end{equation}
     
     \textbf{Bounding $\mathrm{meas}(A_{j,k})$:} We have
     \begin{align}
     \label{eq:bound_g1}
         |g_1(t)|&\ll T^\varepsilon E_1^{-1/2}\Rightarrow \left|\frac{\prod_i (t_1-r_i)}{(t_1+t_2)Y_1(2t_1+\Im(s))^2}-\alpha_1\right |\ll T^\varepsilon E_1^{-1/2}\textnormal{ and}\\
     \label{eq:bound_g2}
         |g_2(t)|&\ll T^\varepsilon E_2^{-1/2}\Rightarrow \left|\frac{\prod_i (t_2+r_i)}{(t_1+t_2)Y_2}-\alpha_2\right |\ll T^\varepsilon E_2^{-1/2}
     \end{align}
     for $\alpha_i\in \{\pm1\}$. Fix such a choice for the $\alpha_i$ and let $A_{j,k}^{\alpha_1,\alpha_2}$ be the corresponding region. As in \cite{BB20}*{Section 15.4}, this set has a finite number of connected components by \cite{Mi64}*{Theorem 3} and this number is bounded in terms of the degrees of the above polynomials and thus bounded independently of all the variables. For $t\in A_{j,k}^{\alpha_1,\alpha_2}$ let $A_{j,k}^{\alpha_1,\alpha_2}(t)$ be the connected component containing $t$. If we can prove that $(t_1\pm u_1,t_2)\notin A_{j,k}^{\alpha_1,\alpha_2}(t)$ for some $u_1$ independent of $t$, we can bound $\mathrm{meas}(A_{j,k}^{\alpha_1,\alpha_2}(t))$ by $|u_1|E_2$. Similarly $(t_1,t_2\pm u_2)\notin A_{j,k}^{\alpha_1,\alpha_2}(t)$ for some $u_2$ independent of $t$ gives the bound $|u_2|E_1$. 
     
     If $(t_1+x,t_2)\in A_{j,k}^{\alpha_1,\alpha_2}(t_1,t_2)$, then
     \begin{align*}
         &T^{\varepsilon}E_1^{-1/2}\gg g_1(t_1+x,t_2)-g_1(t_1,t_2)= x\frac{\partial}{\partial t_1} g_1(t_1,t_2)+O(\frac{x^2}{E_1^2})\\ \textnormal{and }
         &T^{\varepsilon}E_2^{-1/2}\gg g_2(t_1+x,t_2)-g_2(t_1,t_2)= x\frac{\partial}{\partial t_1} g_2(t_1,t_2)+O(\frac{x^2}{B_3})\ .
     \end{align*}
     
     Thus, for $u_1:=T^{2\varepsilon}B_3E_2^{-1/2}$ we have $(t_1\pm u_1,t_2)\notin A_{j,k}^{\alpha_1,\alpha_2}(t_1,t_2)$, which implies
     \begin{equation*}
         \mathrm{meas}(A_{j,k}^{\alpha_1,\alpha_2}(t_1,t_2))\ll T^{2\varepsilon}B_3E_1^{1/2}
     \end{equation*}
     and analogous
     \begin{equation*}
         \mathrm{meas}(A_{j,k}^{\alpha_1,\alpha_2}(t_1,t_2))\ll T^{2\varepsilon}B_3E_2^{1/2}\ .
     \end{equation*}
     Plugging this into $(\ref{eq:bound_J_general})$ gives
     \begin{equation}
         \label{eq:bound_J} J\ll T^\varepsilon\frac{1}{(T+|U_1|)(T+|U_2|)}\Big(\frac{B_3}{B_1B_2}\Big)^{1/2}B_3\min(E_1,E_2)^{1/2}\ .
     \end{equation}
     
     This bound is enough to exclude the remaining non-generic cases:
     
     \textbf{Case 1: $|U_2|\leq T^{1-b}$.} In this case, $B_2\asymp T$, $B_1\geq T^{1-\varepsilon}$, and $B_3\leq B_1T^\varepsilon$. We can use this and $(\ref{eq:bound_J})$ to bound $|U_1|^{-1/2}J$ by
     \[
        |U_1|^{-1/2}T^{\varepsilon}\frac{1}{T+|U_1|}T^{-1}\Big(\frac{U_1}{B_1T}\Big)^{1/2}|U_1||U_2|^{1/2}\ll T^{-3/2-b/4+\varepsilon}\ .
     \]
     Together with the $\mu$ integral this is below the claimed bound.
     
     \textbf{Case 2: $|U_1|\geq T^{1+b}$ and $|U_2|\geq T^{1-b}$.} In this case, $B_1\asymp |U_1|$, $B_2\geq T^{1+\varepsilon}$ and $B_3\leq |U_1|$. We can use this and $(\ref{eq:bound_J})$ to bound $|U_1|^{-1/2}J$ by
     \[
        |U_1|^{-1/2}T^{\varepsilon}|U_1|^{-1}\frac{1}{T+|U_2|}\Big(\frac{1}{B_2}\Big)^{1/2}|U_1|\min(T,|U_2|)^{1/2}\ll T^{-3/2-b/4+\varepsilon}\ .
     \]
     Together with the $\mu$ integral this is below the claimed bound.
     
     \textbf{Case 3: $B_3\leq T^{1-b}$, $|U_1|\leq T^{1+b}$ and $|U_2|\geq T^{1-b}$.} In this case, $B_3\leq T^\varepsilon\min(B_1,B_2)$. We can use this and $(\ref{eq:bound_J})$ to bound $|U_1|^{-1/2}J$ by
     \[
        |U_1|^{-1/2}T^\varepsilon\frac{1}{(T+|U_1|)(T+|U_2|)}B_3^{3/2}(B_1B_2)^{-1/2}\min(B_3,|U_1|,|U_2|)^{1/2}\ .
     \]
     As only $T^{1-b}\ll |U_1|,|U_2|\ll T^{1+b}$ remain, this is maximal for $B_1\asymp B_2\asymp B_3$ and $|U_1|\asymp|U_2|\asymp T$. In which case, we can bound it by $T^{-3/2-b}$ and together with the $\mu$ integral this is below the claimed bound.
     
     \textbf{The nearly generic case:} We are left with the case
     \[
        T^{1-b}\ll |U_1|,|U_2|,B_1,B_2,B_3\ll T^{1+b}\ .
     \]
     In this case, one dimensional integration by parts or stationary phase is not enough. We take a closer look at the condition (\ref{eq:bound_g1}). Rearranging it to isolate $t_2$ gives
     \[
        t_2=\alpha_1\frac{\prod_i(t_1-r_i)-\alpha_1t_1Y_1(2t_1+\Im(s))^2}{Y_1(2t_1+\Im(s))^2}+O(T^\varepsilon E_1^{-1/2}B_3)\ .
     \]
     Using $|\Im(s)|\leq T^{\varepsilon}$ and multiplying out this gives
     \begin{equation}
         \label{eq:formula_t2}
         t_2=(\frac{\alpha_1}{4Y_1}-1)t_1-\frac{1}{2}(r_1^2+r_2^2+r_3^2)t_1^{-1}-r_1r_2r_3t_1^{-2}+O(T^\varepsilon E_1^{-1/2}B_3)\ .
     \end{equation}
     Rearranging the condition (\ref{eq:bound_g2}) leads to
     \[
        \prod_i (t_2+r_i)-\alpha_2(t_1+t_2)Y_2\ll T^\varepsilon E_2^{-1/2}B_3Y_2.
     \]
     Plugging in $(\ref{eq:formula_t2})$, multiplying with $t_1^6$ and taking care of the error terms gives
     \[
        \sum_{i=0}^9 a_it^i \ll T^{\frac{17}{2}+\frac{17}{2}b+\varepsilon}
     \]
     with complex numbers $a_i$ independent of $t_1$. We have
     \begin{align*}
         a_9&= (\frac{\alpha_1}{4Y_1}-1)^3\\
         a_8&=0\\
         a_7&= -(\alpha_1\alpha_2\frac{Y_2}{4Y_1}+\frac{1}{2}(\frac{\alpha_1}{4Y_1}-1)\sum_i r_i^2+\frac{1}{2}(\frac{\alpha_1}{4Y_1}-1)^2\sum_i r_i^2)\ .
     \end{align*}
     We distinguish two cases: If $|\frac{\alpha_1}{4Y_1}-1|\gg T^{-a}$ for $a=12b$, then
     \[
        \prod_i^9|t_1-q_i|\ll \frac{T^{\frac{17}{2}+\frac{17}{2}b+\varepsilon}}{|a_9|}\ll T^{\frac{17}{2}+\frac{17}{2}b-3a+\varepsilon}
     \]
     for some complex numbers $q_i$ independent of $t_1$. This implies
     \[
        \min_i |t_1-q_i|\ll T^{\frac{17}{18}+\frac{17}{18}b-\frac{1}{3}a+\varepsilon}
     \]
     and thus using (\ref{eq:formula_t2}) we can bound $\mathrm{meas}(A_{j,k}^{\alpha_1,\alpha_2}(t_1,t_2))\ll T^{\frac{13}{9}+\frac{22}{9}b-\frac{1}{3}a+\varepsilon}$.
     
     If $|\frac{\alpha_1}{4Y_1}-1|\ll T^{-a}$, then $|a_7|\gg T^{2-2b}$ and we have
     \[
        \prod_i^7|t_1-q^\prime_i|\ll \frac{T^{\frac{17}{2}+\frac{17}{2}b+\varepsilon}}{|a_7|}+\frac{|a_9|T^9}{|a_7|}\ll T^{\frac{13}{2}+\frac{21}{2}b+\varepsilon}+T^{7-3a+b}
     \]
     for some complex numbers $q^\prime_i$ independent of $t_1$. As above we can argue that
     \[
        \mathrm{meas}(A_{j,k}^{\alpha_1,\alpha_2}(t_1,t_2))\ll T^{\frac{10}{7}+3b+\varepsilon}+T^{\frac{3}{2}-\frac{3}{7}a+\frac{23}{14}b+\varepsilon}\ .
     \]
     Plugging in the value for $b$ into all these possible bounds and combining it with (\ref{eq:bound_J_general}) and the missing factor $|U_1|^{1/2}$ and the $\mu$ integral we get the claimed bound or a smaller one in all cases.
\end{enumerate}

\end{proof}

\section{Gau{\ss} sums}
\label{sec:gauss}
When working with the Kuznetsov formula, one often has to bound some Fourier transform of the relevant Kloosterman sums. In the case of the symmetric square, this leads to some form of Gau{\ss} sums, as for example in \cite{BC25}*{Lemma 17}, and \cite{Ve02}*{2.3.3 Analysis of the local sum}. These sums exhibit essentially square root cancellation for square free moduli, but the story becomes more delicate in the case of powerful moduli, and we need the following lemma.

\begin{lemma}
\label{lemma:exp_sum_bound}
Let $D\in\IZ^+$  and let $b,c\in\IZ$ . The following bound holds
\begin{align*}
    \sum_{k\mod D}e\Big(\frac{bk}{D}\Big)S(k^2,c;D)\ll sD^{1+\varepsilon}\ ,
\end{align*}
where $D=s^2f$ with $f$ squarefree. 
\end{lemma}
\begin{proof}
We have
\begin{align*}
    \sum_{\substack{k\mod D}}e\Big(\frac{bk}{D}\Big)S(k^2,c;D)&=\sum_{\substack{x,k\mod D\\ (x,D)=1}}e\Big(\frac{xk^2+bk+c\overline{x}}{D}\Big)\\
    &=\sum_{\substack{x,k\mod D\\ (x,D)=1}}e\Big(\frac{(k^2+bk+c)\overline{x}}{D}\Big)\\
    &=\sum_{\substack{k\mod D}}c_D(k^2+bk+c)\ ,
\end{align*}
where
\[
    c_q(n):=\sum_{\substack{x\mod q\\(x,q)=1}}e(\frac{nx}{q})
\]
is the Ramanujan sum. Here, we changed variables from $k$ to $\overline{x}k$ in the second equality. Because $c_D(n)$ is multiplicative in $D$ and only depends on $n\mod D$, the whole sum can be factored over the prime factors of $D$.

Assume $D=p^\lambda$ with $p$ an odd prime.
Then, we have
\[
    c_D(n)=\begin{cases}-p^{\lambda-1} & \textnormal{ if } v_p(n)=\lambda-1\\
    p^{\lambda}-p^{\lambda-1} & \textnormal{ if } v_p(n)\geq\lambda\\
    0 & \textnormal{ else}\ .
    \end{cases}
\]
In our case $n=k^2+bk+c$. Thus, we have to count how often $v_p(k^2+bk+c)=\lambda$ and how often $v_p(k^2+bk+c)=\lambda -1$ for $k\mod p^\lambda$. If $\lambda=1$, there are at most $2$ values for $k$ such that $k^2+bk+c\equiv 0 \mod p$, and in the remaining cases $v_p(k^2+bk+c)=0$. This leads to a bound of $2p$. From now on, assume $\lambda>1$.
Let $\Delta:=b^2-4c$ be the discriminant of $k^2+bk+c$. After completing the square and a change of variables, we are interested in $v_p(k^2-\Delta)$. We distinguish three cases for $v_p(\Delta)$.

\textbf{Case 1: $v_p(\Delta)\leq \lambda-1$ and $v_p(\Delta)$ even.} In this case, there are either $2p^{v_p(\Delta)/2}$ or $0$ many solutions $\mod p^{v_p(\Delta)}$ to $k^2-\Delta\equiv 0\mod p^{v_p(\Delta)}$. Let $a$ be this number. By Hensel's Lemma, we can lift these to the same number of solutions $\mod p^{\lambda-1}$ to $k^2-\Delta\equiv 0\mod p^{\lambda -1}$. Going to $\mod p^\lambda$, each solution has a unique lift such that  $k^2-\Delta\equiv 0\mod p^{\lambda}$ and $p-1$ lifts such that  $v_p(k^2-\Delta)=\lambda-1$. In total, the sum is $a(p^\lambda-p^(\lambda-1)-(p-1)p^{\lambda-1})=0$. 

\textbf{Case 2: $v_p(\Delta)\leq \lambda-1$ and $v_p(\Delta)$ odd.} In this case, $v_p(k^2)\neq v_p(\Delta)$ and thus $v_p(k^2-\Delta)=\min(v_p(\Delta),v_p(k^2))$. This can only be $\lambda-1$ if $v_p(\Delta)=\lambda-1$, and it can never be $\lambda$. In the case of $v_p(\Delta)=\lambda-1$, we get $p^{\lambda/2}$ choices for $k$ such that  $v_p(k^2-\Delta)=\lambda-1$, and in total, the sum is bounded by $p^{\lambda/2}p^{\lambda-1}$, which is bounded by $sD$.

\textbf{Case 3: $v_p(\Delta)\geq \lambda$}. In this case, we are just interested in $v_p(k^2)$. If $\lambda$ is even, this cannot be $\lambda-1$, and we get $p^{\lambda/2}$ choices for $k$ such that  $v_p(k^2)\geq \lambda$, and the sum is bounded by $p^{\frac{3}{2}\lambda}=sD$. If $\lambda$ is odd, there are $\varphi(p^{(\lambda-1)/2})$ choices for $k$ such that  $v_p(k^2)=\lambda-1$ and $p^{(\lambda-1)/2}$choices for $k$ such that $v_p(k^2)\geq \lambda$. This leads to a bound of $p^{\lambda+(\lambda-1)/2}=sD$.

Now, let $D=2^\lambda$. The argument is almost the same as before with some small changes. If $b$ is odd, we can always start $\mod 2$ and use Hensel's Lemma to lift all solutions uniquely. The bound in this case is $D$. If $b=2B$ is even, we can complete the square and look at $k^2-(B^2-c)$. We can then again distinguish by $v_2(B^2-c)$ and get the same bounds with at most an additional factor $2$ coming from the fact that there are $4$ solutions to $x^2\equiv 1$ modulo $8$ and not $2$.

Multiplying all local bounds together and absorbing the constants into the $\varepsilon$ gives the desired bound.
\end{proof}

\section{Proof of Theorem 1}
\label{sec:proofOfThm1}
We want to compute
\[
    \sum_{\pi\in \mathcal{B}} \overline{A_\pi(m_1,m_2)}L(1/2,\pi,\mathrm{sym}^2)\cdot \frac{h(\mu_\pi)}{\CN(\pi)}\ ,
\]
where $h\in \tilde{\CH}(A_0,\tilde{\mu})$.
Applying the approximate functional equation from Theorem \ref{thm:approxfunctionalequation} gives
\begin{align*}
    &\sum_{\pi\in \mathcal{B}} \overline{A_\pi(m_1,m_2)}L(1/2,\pi,\mathrm{sym}^2)\cdot  \frac{h(\mu_\pi)}{\CN(\pi)}\\
    =& \sum_{\pi\in \mathcal{B}}\zeta(\frac{3}{2})\cdot \sum_{b,c}\Big( \frac{A_\pi(b^2,c^2)\overline{A_\pi(m_1,m_2)}}{(b^2c)^{1/2}}V_\pi(\frac{b^2c}{T^{3+\alpha}})-\frac{A_\pi(m_2,m_1)\overline{A_{\pi}(b^2,c^2)}}{(b^2c)^{1/2}}W_\pi(b^2cT^{3+\alpha})\Big)\frac{h(\mu_\pi)}{\CN(\pi)}\\
    =&\zeta(\frac{3}{2})\cdot \sum_{b,c}\frac{1}{(b^2c)^{1/2}}\Big( \sum_{\pi\in \mathcal{B}}A_\pi(b^2,c^2)\overline{A_\pi(m_1,m_2)}\frac{h(\mu_\pi)}{\CN(\pi)}V_\pi(\frac{b^2c}{T^{3+\alpha}})\\&\quad\quad-\sum_{\pi\in \mathcal{B}}A_\pi(m_2,m_1)\overline{A_{\pi}(b^2,c^2)}\frac{h(\mu_\pi)}{\CN(\pi)}W_\pi(b^2cT^{3+\alpha})\Big)\ .
\end{align*}

For $b^2c\gg T^{3+\alpha+\epsilon}$ respectively $b^2c\gg T^{3-\alpha+\epsilon}$, we can use that $A_\pi(m,n)\ll |mn|^{O(1)}$ and that $\frac{1}{\CN(\pi)}\ll \|\mu_\pi\|^{\varepsilon}$; see Lemma \ref{lemma:normfactor_bounds}, and shift the contour in $V$ respectively $W$ to the far right, resulting in a bound of $T^{-B}$ for some large $B>0$. We can thus truncate the $b,c$ sums at $b^2c\ll T^{3+\alpha+\epsilon}$ respectively $b^2c\ll T^{3-\alpha+\epsilon}$ at the cost of a negligible error.

\subsection{Adding the rest of the spectrum}
\label{sec:Symmetric_AddingSpectrum}
To apply the Kuznetsov formula, we need to add the continuous part of the spectrum. As we are in level one, there are no old forms, and we only have to add the minimal and the maximal Eisenstein series. For more details on this part of the spectrum, we refer to \cite{Bu16}.

\subsubsection{The minimal Eisenstein series}
For the minimal Eisenstein series, we need to bound
\begin{align*}
    \sum_{\substack{b,c\\b^2c\ll T^{3+ \alpha+\varepsilon}}}\frac{1}{(b^2c)^{1/2}} \sum_{j=0}^2 \int_{\Re\mu=0}\frac{h(\mu)V_\mu(b^2c T^{-(3+\alpha)})}{\prod_{1\leq i<l\leq 3}|\zeta(1+\mu_i-\mu_l)|}\overline{B_\mu^{(j)}(m_1,m_2)}B_\mu^{(j)}(b^2,c^2)\mathrm{d}\mu
\end{align*}
and 
\begin{align*}
    \sum_{\substack{b,c\\b^2c\ll T^{3- \alpha+\varepsilon}}}\frac{1}{(b^2c)^{1/2}} \sum_{j=0}^2 \int_{\Re\mu=0}\frac{h(\mu)W_\mu(b^2c T^{3+\alpha})}{\prod_{1\leq i<l\leq 3}|\zeta(1+\mu_i-\mu_l)|}\overline{B_\mu^{(j)}(b^2,c^2)}B_\mu^{(j)}(m_2,m_1)\mathrm{d}\mu\ .
\end{align*}
We can bound the Fourier coefficients $B_\mu^{(j)}(m_1,m_2)$ by $(m_1m_2)^\varepsilon$ and thus by $T^\varepsilon$ in all cases. By the rapid decay of $h$ away from $\tilde{\mu}$, we can truncate the $\mu$-integral to a ball of radius $T^{\delta+\varepsilon}$ around the $w(\tilde{\mu})$ for $w\in \mathcal{W}$. Using Lemma \ref{lemma:bound_VW} with $A$ close to $0$ and (\ref{eq:bound_h_individuel}), we can bound the $h$ and $V$ respectively $W$ terms by $T^\varepsilon$. Furthermore, we can bound the inverse $\zeta$-factors by $T^\varepsilon$ using the Prime Number Theorem.
This bounds the integral by $T^{2\delta+\varepsilon}$, and the sum over $b,c$ adds $T^{\frac{3}{2}\pm\alpha/2}$ to this, which is smaller than the claimed error term for small enough $\delta,\alpha$.

\subsubsection{The maximal Eisenstein series}
For the maximal Eisenstein series, we need to bound
\begin{align*}
    \sum_{\substack{b,c\\b^2c\ll T^{3+ \alpha+\varepsilon}}}\frac{1}{(b^2c)^{1/2}} \sum_\phi \int_{\Re\mu_1=0}\frac{h(\mu^\prime)V_{\mu^\prime}(b^2c T^{-(3+\alpha)})}{L(1,\phi,\mathrm{Ad})|L(1+3\mu_1,\phi)|^2}\overline{A_{\phi,\mu_1}(m_1,m_2)}A_{\phi,\mu_1}(b^2,c^2)\mathrm{d}\mu
\end{align*}
and
\begin{align*}
    \sum_{\substack{b,c\\b^2c\ll T^{3- \alpha+\varepsilon}}}\frac{1}{(b^2c)^{1/2}} \sum_\phi \int_{\Re\mu_1=0}\frac{h(\mu^\prime)W_{\mu^\prime}(b^2c T^{3+\alpha})}{L(1,\phi,\mathrm{Ad})|L(1+3\mu_1,\phi)|^2}\overline{A_{\phi,\mu_1}(b^2,c^2)}A_{\phi,\mu_1}(m_2,m_1)\mathrm{d}\mu\ ,
\end{align*}
where the $\phi$ sum goes over an orthogonal basis of cusp forms for $\GL_2$ of level one with spectral parameter $\mu_\phi$ and $\mu^\prime=(\mu_1+\mu_\phi,\mu_1-\mu_\phi,-2\mu_1)$.

We can bound $h$ and $V$ respectively $W$ in the same way as above. We can bound the Fourier coefficients by $A_{\phi,\mu_1}(m_1,m_2)\ll (m_1m_2)^{\epsilon}$ because there are no exceptional forms of level 1 for $\GL_2$ over $\IQ$ \cite{BS07}. We can bound the inverse of the adjoint $L$-function by $T^\varepsilon$ using \cite{GHL94}*{Main Theorem} and the inverse of the other $L$-function by $T^\varepsilon$ using \cite{Sa04}*{(4)}. Using the rapid decay away from $\tilde{\mu}$ of $h$, we can truncate the integral to $\min_{w\in W}|\mu_1-w(\tilde{\mu})_3|\leq T^{\delta+\varepsilon}$. Furthermore, we can truncate the $\phi$-sum to $|\mu_1+\mu_\phi-w(\tilde{\mu})_1|\leq T^{\delta+\varepsilon}$. Using Weyl's law \cite{Mue08}*{Theorem 4.2}, we can bound the number of $\phi$ satisfying this by $T^{1+\delta+\varepsilon}$. Bringing it all together, we get the bound $T^{\frac{5}{2}+\alpha/2+2\delta+\varepsilon}$, which is smaller than the claimed error term for small enough $\delta,\alpha$.

\begin{rem}
    If one wants to work with forms of level $N>1$, one has to account for exceptional forms. Just plugging in the Kim-Sarnak bound of $a_f(m)\ll m^{\frac{7}{64}}$ is not quite enough and leads to an error term bigger than $T^3$. In that case, we would need to also use a density theorem, as in \cite{Hu18}, to control the number of these exceptions. Alternatively, we can use the fact that we required $\tilde{\mu}$ to be away from the walls, which implies that $h(\mu^\prime)$ would be negligible for an exceptional form $\phi$.
\end{rem}

\subsection{Test functions and applying Kuznetsov}
Before applying the Kuznetsov formula to both sums with the two test functions
\[
    V_{b,c}(\mu):=h(\mu)V_\mu(\frac{b^2c}{T^{3+\alpha}}),\quad W_{b,c}(\mu):=h(\mu)W_\mu(b^2cT^{3+\alpha})
\]
we make a few observations. We have
\begin{equation}
    \label{eq:V_uint}
    {V_{b,c}}=\int_{(\epsilon)}\frac{e^{u^2}}{u}\frac{\zeta(3(\frac{1}{2}+u))}{\zeta(\frac{3}{2})}\left(\frac{b^2c}{T^{3+\alpha}}\right)^{-u}{\tilde{V}_u}\frac{\mathrm{d}u}{2\pi i}
\end{equation}
respectively
\begin{equation}
\label{eq:W_uint}
{W_{b,c}}=\int_{(\epsilon)}\frac{e^{u^2}}{u}\frac{\zeta(3(\frac{1}{2}+u))}{\zeta(\frac{3}{2})}\left(b^2cT^{3+\alpha}\right)^{-u}{\tilde{W}_u}\frac{\mathrm{d}u}{2\pi i}\ .
\end{equation}
With
\[
    \tilde{V}_u(\mu):=h(\mu)\frac{\mathcal{G}(u,\mu)}{\mathcal{G}(0,\mu)},\quad \tilde{W}_u(\mu)= h(\mu)\frac{\mathcal{G}(-u,\mu)}{\mathcal{G}(0,\mu)}\prod_{1\leq i\leq j\leq 3}\frac{\Gamma_{\IR}(\frac{1}{2}+u+\mu_i+\mu_j)}{\Gamma_{\IR}(\frac{1}{2}-u-\mu_i-\mu_j)}\ .
\]
We later want to apply the results from Section \ref{section:LemmaWeightFunctions}. To do this, we will need the following two Lemmata:
\begin{lemma}
\label{lemma:testfunctions_pre}
For $h\in\tilde{\CH}(A_0,\tilde{\mu})$, $u=\varepsilon +it$ with $|t|\leq T^\varepsilon$, we have
\[
        \tilde{V}_u(\mu)\in \CH(A_0,\tilde{\mu})\ ,\quad \tilde{W}_u(\mu)\in \CH(A_0) \,
    \]
    and $\tilde{W}_u(\mu)$ satisfies (\ref{eq:rapid_decay_awayfrommu}) and (\ref{eq:BoundSpec_h}), and $\tilde{W}_u(\mu)\exp(-h_\infty(\mu))$ satisfies (\ref{eq:bound_derivatives_h}).
\end{lemma}
\begin{proof}
    For $\mu$ in generic position $\frac{\mathcal{G}(\pm(\varepsilon+it),\mu)}{\mathcal{G}(0,\mu)}$ is of size $1$, it only has poles where $h$ has zeros because of (\ref{eq:zerosLinf}), and it is of polynomial size in $\|\mu\|$ away from these points. The factor
\[
    \prod_{1\leq i\leq j\leq 3}\frac{\Gamma_{\IR}(\frac{1}{2}+(\varepsilon+it)+\mu_i+\mu_j)}{\Gamma_{\IR}(\frac{1}{2}-(\varepsilon+it)-\mu_i-\mu_j)}
\]
    is also of polynomial size in $\|\mu\|$, and all its poles are canceled by $\mathcal{G}(-(\varepsilon+it),\mu)$. For $\Re \mu=0$, it is bounded by $\|\mu\|^{6\varepsilon}$\ .
    
    Thus, $\tilde{V}_u(\mu)$ and $\tilde{W}_u(\mu)$ both satisfy (\ref{eq:rapid_decay_awayfrommu}), (\ref{eq:BoundSpec_h}), (\ref{eq:zerosKuznetsov}), and Lemma \ref{lemma:intpartsVW} implies that $\tilde{V}_u(\mu)$ and $\tilde{W}_u(\mu)\exp(-h_\infty(\mu))$ satisfy (\ref{eq:bound_derivatives_h}).
\end{proof}

\begin{lemma}
\label{lemma:testfunctions}
For $h\in\tilde{\CH}(A_0,\tilde{\mu})$, $b,c\in O(T^B)$ we have
\[
        V_{b,c}(\mu)\in \CH(A_0,\tilde{\mu})\ ,\quad W_{b,c}(\mu)\in \CH(A_0) \,
    \]
    and $W_{b,c}(\mu)$ satisfies (\ref{eq:rapid_decay_awayfrommu}) and (\ref{eq:BoundSpec_h}), and $W_{b,c}(\mu)\exp(-h_\infty(\mu))$ satisfies (\ref{eq:bound_derivatives_h}).
\end{lemma}
\begin{proof}
Using (\ref{eq:V_uint}) respectively (\ref{eq:W_uint}) we realize that the integral for $|\Im (u)|\geq T^\varepsilon$ is Weyl group invariant, holomorphic, and in $O_B(T^{-B})$ for all $B>0$, which also holds for all derivatives in the $\mu_i$, due to the rapid decay of $e^{u^2}$. We can therefore truncate the integral at $|\Im (u)|\leq T^\varepsilon$ without affecting any of the desired properties. For the remaining integral, we use Lemma \ref{lemma:testfunctions_pre},
\[
    \frac{e^{u^2}}{u}\frac{\zeta(3(\frac{1}{2}+u))}{\zeta(\frac{3}{2})}\left(\frac{b^2c}{T^{3+\alpha}}\right)^{-u}\ll_\varepsilon T^\varepsilon,\quad \frac{e^{u^2}}{u}\frac{\zeta(3(\frac{1}{2}+u))}{\zeta(\frac{3}{2})}\left(b^2cT^{3+\alpha}\right)^{-u}\ll_\varepsilon T^\varepsilon 
\]
and that $\tilde{V}_u$ and $\tilde{W}_u$ and all their derivatives in the $\mu_i$ depend only polynomially on $\Im (u)$, to see that the properties are also preserved for this part of the integral. 
\end{proof}
We will often use Lemma \ref{lemma:testfunctions} implicitly when using the results from Section \ref{section:LemmaWeightFunctions} and not mention it every time.

After applying Kuznetsov to both sums, we are left with $8$ terms
\[
    \sum_{\pi\in \mathcal{B}} \overline{A_\pi(m_1,m_2)}L(1/2,\pi,\mathrm{sym}^2)  \frac{h(\mu_\pi)}{\CN(\pi)} = \zeta(\frac{3}{2})\big(\Delta^{(1)}-\Delta^{(2)}+\Sigma_4^{(1)}-\Sigma_4^{(2)}+\Sigma_5^{(1)}-\Sigma_5^{(2)}+\Sigma_6^{(1)}-\Sigma_6^{(2)} \big)+O(T^{3-\kappa}) \ .
\]
\subsection{The diagonal}
We have
\begin{align*}
    192\pi^5\Delta^{(1)}=& \sum_{b,c}\delta_{\substack{b^2=m_1\\c^2=m_2}} \frac{1}{(b^2c)^{1/2}}\int_{\Re\mu=0}h(\mu)V_\mu(\frac{b^2c}{T^{3+\alpha}})\mathrm{spec}(\mu)\mathrm{d}\mu\\
    =& \delta_{\substack{m_1=\square\\m_2=\square}}\frac{1}{(m_1^2m_2)^{1/4}}\int_{\Re\mu=0}h(\mu)V_\mu(\frac{m_1m_2^{1/2}}{T^{3+\alpha}})\mathrm{spec}(\mu)\mathrm{d}\mu\\
    =& \delta_{\substack{m_1=\square\\m_2=\square}}\frac{1}{(m_1^2m_2)^{1/4}}\Big(\int_{\Re\mu=0}h(\mu) \int_{(-A)}\frac{T^{(3+\alpha)u}}{(m_1^2m_2)^{u/2}}\frac{e^{u^2}}{u}\frac{\zeta(3(\frac{1}{2}+u))}{\zeta(\frac{3}{2})}\frac{\mathcal{G}(u,\mu)}{\mathcal{G}(0,\mu)}\frac{\mathrm{d}u}{2\pi i} \mathrm{spec}(\mu)\mathrm{d}\mu\\&\quad+\int_{\Re\mu=0}h(\mu)\mathrm{spec}(\mu)\mathrm{d}\mu\Big)\\
    =& \delta_{\substack{m_1=\square\\m_2=\square}}\frac{1}{(m_1^2m_2)^{1/4}}\mathrm{vol}_\spec(h) + O_A(T^{-A(3+\alpha-\frac{3}{2}\eta)})\ ,
\end{align*}
where we shifted the contour to $-A$ and picked up the pole at $u=0$.

We also have
\begin{align*}
    192\pi^5\Delta^{(2)}=& \sum_{b,c}\delta_{\substack{b^2=m_2\\c^2=m_1}} \frac{1}{(b^2c)^{1/2}}\int_{\Re\mu=0}h(\mu)W_\mu(b^2cT^{3+\alpha})\mathrm{spec}(\mu)\mathrm{d}\mu\\
    =& \delta_{\substack{m_1=\square\\m_2=\square}}\frac{1}{(m_2^2m_1)^{1/4}}\int_{\Re\mu=0}h(\mu)W_\mu((m_2^2m_1)^{1/2}T^{3+\alpha})\mathrm{spec}(\mu)\mathrm{d}\mu\\
    =&\delta_{\substack{m_1=\square\\m_2=\square}}\frac{1}{(m_2^2m_1)^{1/4}}\Big(\int_{\Re\mu=0}h(\mu)\prod_{1\leq i\leq j\leq 3}\frac{\Gamma_{\IR}(\frac{1}{2}+\mu_i+\mu_j)}{\Gamma_{\IR}(\frac{1}{2}-\mu_i-\mu_j)}\mathrm{spec}(\mu)\mathrm{d}\mu \\
    &+ \int_{\Re\mu=0}h(\mu) \int_{(-A)}T^{-(3+\alpha)u}(m_2^2m_1)^{-u/2}\\&\quad\cdot\prod_{1\leq i\leq j\leq 3}\frac{\Gamma_{\IR}(\frac{1}{2}+\mu_i+\mu_j)}{\Gamma_{\IR}(\frac{1}{2}-\mu_i-\mu_j)}\frac{e^{u^2}}{u}\frac{\zeta(3(\frac{1}{2}+u))}{\zeta(\frac{3}{2})}\frac{\mathcal{G}(-u,\mu)}{\mathcal{G}(0,\mu)}\frac{\mathrm{d}u}{2\pi i} \mathrm{spec}(\mu)\mathrm{d}\mu\Big)\ .
\end{align*}
The second summand is in $O_A(T^{-A(3-\alpha)})$. To bound the first summand, we first replace $h$ up to a negligible error with a Weyl group invariant real analytic function $\tilde{h}$ that has support in $\min_{w\in W}\|\mu-w(\tilde{\mu})\|\leq T^{\delta+\varepsilon}$ and also satisfies all conditions for being in $\CH(A_0,\tilde{\mu})$ except for holomorphicity.
Then we observe that
\[
    \tilde{h}(\mu)\prod_{1\leq i\leq j\leq 3}\frac{\Gamma_{\IR}(\frac{1}{2}+\mu_i+\mu_j)}{\Gamma_{\IR}(\frac{1}{2}-\mu_i-\mu_j)}\mathrm{spec}(\mu)
\]
is Weyl group invariant. Thus, we can reduce to the Weyl chambers with $\Im(\mu_1),\Im(\mu_2)>0>\Im(\mu_3)$. Assume without loss of generality that $\tilde{\mu}$ lies in this chamber. As in the proofs of Lemma \ref{lemma:fine_bound_inttransfrom_voronoi} and Lemma \ref{lemma:fine_bound_inttransfrom_long}, we use Lemma \ref{lemma:derivative_hinf} to see that
\[
    |\frac{\partial}{\partial\Im(\mu_2)}h_\infty(\mu)|=C+O(T^{-1+\delta+\varepsilon}),\quad \frac{\partial^j}{\partial\Im(\mu_2)^j}h_\infty(\mu)\ll T^{-j+1+\delta+\varepsilon}
\]
for $\|\mu-\tilde{\mu}\|<T^{\delta+\varepsilon}$, $j\geq 2$, and $C$ dependent on $\tilde{\mu}$ and bounded from below and above in terms of $c_1$ and $c_2$.

Thus, we can apply Lemma \ref{lemma:intbyparts} with $U=T^{\delta}$, $R=C$, $Q=Y=T$, and $X=T^B$ for some fixed $B$ and conclude that this contribution is negligible.

\subsection{The Voronoi elements}
\subsubsection{Bounding $\Sigma_{4}^{(1)}$}
The contribution of $\Sigma_{4}^{(1)}$ is
\begin{align*}
    &\sum_{\substack{b,c\\b^2c\ll T^{3+\alpha+\varepsilon}}}(b^2c)^{-1/2}\sum_{\epsilon=\pm}\sum_{\substack{D_2|D_1\\ m_2D_1=b^2D_2^2}}\frac{\tilde{S}(-\epsilon c^2,m_2,m_1,D_2,D_1)}{D_1D_2}\Phi_{w_4}^{V_{b,c}}\left(\frac{\epsilon m_1m_2c^2}{D_1D_2}\right)\\
    =&\sum_{\substack{b,c\\b^2c\ll T^{3+\alpha+\varepsilon}}}(b^2c)^{-1/2}\sum_{\epsilon=\pm}\sum_{\substack{D_2\\ m_2\mid b^2D_2}}\frac{\tilde{S}(-\epsilon c^2,m_2,m_1,D_2,D_2^2b^2/m_2)}{D_2^3b^2}m_2\Phi_{w_4}^{V_{b,c}}\left(\frac{\epsilon c^2m_1m_2^2}{D_2^3b^2}\right)\ .
\end{align*}
As argued in Section \ref{sec:roughboundsabsconv}, we can truncate the $D_2$ sum at $T^{O(1)}$ at the cost of a negligible error.

Lemma \ref{lemma:fine_bound_inttransfrom_voronoi} (a) implies that the contribution of 
\[
    \Big|\frac{c^2m_1m_2^2}{b^2D_2^3}\Big|\leq T^{3-\varepsilon}
\]
is negligible, and we can therefore truncate the $D_2$ sum further at
\[
    D_2\leq \left(\frac{c}{b}\right)^{2/3}|m_1m_2^2|^{1/3}T^{-1+\varepsilon}\ll \left(\frac{c}{b}\right)^{2/3}T^{-1+\eta+\varepsilon}\ .
\]
We now sum over $b$ and $c$ in dyadic ranges $B$ and $C$ using a partition of unity given by a smooth, compactly supported function $w$ with bounded derivatives. By definition, the Kloosterman sum only depends on $c$ modulo $D_2$, and we can perform Poisson summation in $c$ modulo $D_2$, resulting in
\begin{align*}
    &\sum_{\substack{B,C\\ B^2C\ll T^{3+\alpha+\varepsilon}}}\sum_{\substack{b,c}}w(\frac{b}{B})b^{-1}\sum_{\epsilon=\pm}\sum_{\substack{D_2\ll \left(\frac{C}{B}\right)^{2/3}T^{-1+\eta+\varepsilon}\\m_2\mid b^2D_2}}\sum_{k\mod D_2}e(\frac{ck}{D_2})\frac{\tilde{S}(-\epsilon k^2,m_2,m_1,D_2,D_2^2b^2/m_2)}{D_2^4b^2}m_2\\
    &\cdot\int_\IR y^{-1/2}w(\frac{y}{C})\Phi_{w_4}^{V_{b,y}}\left(\frac{\epsilon y^2m_2}{D_2^3b^2}\right)e(\frac{cy}{D_2})\mathrm{d}y\ .
\end{align*}

\subsubsection{The integral}
After a change of variables, we need to bound
\[
    I:=C^{1/2}\int_\IR y^{-1/2}w(y)\Phi_{w_4}^{V_{b,y}}\left(\frac{\epsilon y^2C^2m_2}{D_2^3b^2}\right)e(\frac{cCy}{D_2})\mathrm{d}y\ .
\]
We want to apply Lemma \ref{lemma:intbyparts} for $c$ nonzero with the phase function $H(y)=\frac{cCy}{D_2}$. Using Lemma \ref{lemma:fine_bound_inttransfrom_voronoi} (b) to bound the derivatives of $\Phi_{w_4}^{V_{b,y}}$, we get
\[
    R=\frac{cC}{D_2}\textnormal{ and }U\gg T+\left|\frac{C^2m_1m_2^2}{B^2D_2^3}\right|^{1/3}\ .
\]
We also have $Q=T^{B}$ for any $B$ as the higher derivatives of $H$ are all $0$.
Thus, the integral is negligible if $RU\gg T^\varepsilon$, meaning
\[
    |c|\geq |m_1m_2^2|^{1/3}(CB^2)^{-1/3}T^\varepsilon+ \frac{D_2}{C}T^{1+\varepsilon}\ll |m_1m_2^2|^{1/3}(CB^2)^{-1/3}T^\varepsilon\ . 
\]
For the $D_2$ sum to be nonempty, we need $C\gg T^{3/2-\eta-\varepsilon}$, and thus the integral is always negligible if $c$ is nonzero. 

For $c=0$, we have
\[
    I=C^{1/2}\int_\IR y^{-1/2}w(y)\Phi_{w_4}^{V_{b,y}}\left(\frac{\epsilon y^2C^2m_1m_2^2}{D_2^3b^2}\right)\mathrm{d}y\ .
\]

The test function depends on $y$. Thus, we need to open up the definition of $\Phi_{w_4}^{V_{b,y}}$, truncate the $u$-integral at $|\Im (u)|\leq T^\varepsilon$, and then use Lemma \ref{lemma:testfunctions_pre} to apply Lemma \ref{lemma:bound_fourier_voronoi} with $\Xi=\frac{\epsilon C^2m_2}{D_2^3b^2}$, $s=\frac{1}{2}+u$, and the test function $V_{b,c}(\mu)$. This implies that the integral is negligible except if
\[
    \big|\frac{\epsilon C^2m_1m_2^2}{D_2^3b^2}\big|\asymp_\varepsilon T^3\Leftrightarrow D_2\asymp_\varepsilon C^{2/3}|m_1m_2^2|^{1/3}b^{-2/3}T^{-1}.
\]
In this case, we estimate the integral by $T^{3/2+2\delta+\varepsilon}C^{1/2}$.

\subsubsection{The exponential sum}
Using that $c=0$, we want to estimate
\begin{align*}
    &\sum_{k\mod D_2}e(\frac{ck}{D_2})\tilde{S}(-\epsilon k^2,m_2,m_1,D_2,D_2^2b^2/m_2)\\
    =&D_2\sum_{\substack{k,x_1\mod D_2\\ x_2\mod b^2D_2/m_2\\ (x_1,D_2)=(x_2,b^2D_2/m_2)=1}}e\left(m_2\frac{\overline{x_1}x_2}{D_2}+m_1m_2\frac{\overline{x_2}}{b^2D_2}-\frac{\epsilon k^2x_1}{D_2}\right)\ .
\end{align*}
We first change $k$ to $\overline{x}_1k$, then $k$ to $x_2k$, and $\overline{x}_1$ to $\overline{x}_2^2\overline{x}_1$. Even though $x_2$ and $x_1,k$ have different moduli, we can always pick representatives for $x_2$ and $\overline{x}_2$ that are also invertible modulo $D_2$. This leads to

\begin{align*}
    &D_2\sum_{\substack{k,x_1\mod D_2\\ x_2\mod b^2D_2/m_2\\ (x_1,D_2)=(x_2,b^2D_2/m_2)=1}}e\left(\frac{\overline{x_1}(-\epsilon k^2+m_2\overline{x}_2)}{D_2}+m_1m_2\frac{\overline{x_2}}{b^2D_2}\right)
\end{align*}

If $D_2\equiv 2\mod 4$, the $k$ sum is zero. Else, it is of size $D_2^{1/2}$ and only depends on $\overline{x}_1$ by $\left(\frac{x_1}{D_2}\right)$. We get that the sum is bounded by
\begin{align*}
    D_2^{3/2}\sum_{\substack{x_1\mod D_2\\ x_2\mod b^2D_2/m_2\\ (x_1,D_2)=(x_2,b^2D_2/m_2)=1}}e\left(m_2\frac{(m_1+b^2\overline{x}_1^2)\overline{x_2}}{b^2D_2}\right)
    =D_2^{3/2}\sum_{\substack{x_1\mod D_2\\ (x_1,D_2)=1}}c_{b^2D_2/m_2}(m_1+b^2\overline{x}_1^2)\ ,
\end{align*}
where 
\[
    c_q(n):=\sum_{\substack{x\mod q\\(x,q)=1}}e(\frac{nx}{q})
\]
is again the Ramanujan sum.

If $b=m_1=m_2=1$, we could proceed as in Lemma \ref{lemma:exp_sum_bound}. To get into this position, we use the fact that the Ramanujan sum is multiplicative and approach the problem locally. 
Recall that we have for $p$ prime
\[
    c_{p^\lambda}(n)=\begin{cases}-p^{\lambda-1} & \textnormal{ if } v_p(n)=\lambda-1\\
    p^{\lambda}-p^{\lambda-1} & \textnormal{ if } v_p(n)\geq\lambda\\
    0 & \textnormal{ else}\ .
    \end{cases}
\]
Let $D_2^\prime=D_2/(D_2,m_2)$ and $m_2^\prime=m_2/(m_2,D_2)$. 
Let $b^2=\tilde{b}^2(b^2,(m_1m_2^\prime)^2)$. If $p\mid \tilde{b}$, then the Ramanujan sum is zero at that place. Thus, $\tilde{b}=1$ and $b\mid m_1m_2^\prime$. Furthermore, let $D_2^{\prime\prime}=D_2^{\prime}/(D_2^\prime,m_1^\infty)$. This implies $(D_2^{\prime\prime},b)=1$. We now factorize the Ramanujan sum over the moduli $D_2^{\prime\prime}$ and $(D_2^\prime,m_1^\infty)b^2/m_2^\prime$. We get
\begin{align*}
    D_2^{3/2}(D_2,m_2)\sum_{\substack{x_1\mod D^{\prime\prime}_2\\ (x_1,D^{\prime\prime}_2)=1}}c_{D^{\prime\prime}_2}(m_1+b^2\overline{x}_1^2) \cdot \sum_{\substack{x_1\mod (D^{\prime}_2,m_1^\infty)\\ (x_1,(D^{\prime}_2,m_1^\infty))=1}}c_{b^2(D^{\prime}_2,m_1^\infty)/m_2^\prime}(m_1+b^2\overline{x}_1^2)\ .
\end{align*}
For the sum modulo $D^{\prime\prime}_2$, we can argue as in the proof of Lemma \ref{lemma:exp_sum_bound}, with the difference that $x_1$ and the discriminant are both coprime to the modulus. Thus, there are at most 2 (or 4 if $p=2$) solutions modulo $p^\lambda$ of 
\[
    m_1+b^2\overline{x}_1^2\equiv 0\mod p^\lambda 
\]
and at most $2p$ solutions modulo $p^\lambda$ of 
\[
    m_1+b^2\overline{x}_1^2\equiv 0\mod p^{\lambda-1}\ .
\]
Thus, we can bound the $x_1$ sum in this case by $D_2^{\prime\prime}$.

For the Ramanujan sum $c_{b^2(D^{\prime}_2,m_1^\infty)/m_2^\prime}(m_1+b^2\overline{x}_1^2)$, we distinguish a few cases. For $p\mid (D^{\prime}_2,m_1^\infty)$, we have either $v_p(D^{\prime}_2)>v_p(m_1)$, in which case the sum is zero, or $v_p(D^{\prime}_2)\leq v_p(m_1)$, in which case we estimate the Ramanujan sum trivially by $p^{v_p(b^2(D_2^\prime,m_1)/m_2^\prime)}$ and the length of the $x_1$ sum by $(D_2^\prime,m_1)$. For $p\nmid (D^{\prime}_2,m_1^\infty)$, we estimate the sum trivially by $p^{v_p(b^2/m_2^\prime)}$. In total, this leads to the bound
\[
     D_2^{3/2}(D_2,m_2)D_2^{\prime\prime}(D_2^\prime,m_1)^2 b^2/m_2^\prime\ll D_2^{5/2}m_1^4 b^2\ .
\]

\subsubsection{Bringing it all together}
Applying the bounds for the integral and exponential sum gives
\begin{align*}
    &T^{3/2+2\delta+\varepsilon}m_1^4 m_2\sum_{\substack{B,C\\ B^2C\ll T^{3+\alpha+\varepsilon}}}C^{1/2}\sum_{\substack{b,c}}w(\frac{b}{B})b^{-1}\sum_{\epsilon=\pm}\sum_{\substack{D_2\asymp_\varepsilon \left(\frac{C}{B}\right)^{2/3}|m_1m_2^2|^{1/3}T^{-1}\\m_2\mid b^2D_2}}D_2^{-3/2}\\
    \ll&\ T^{3/2+2\delta+\varepsilon}m_1^4 m_2\sum_{\substack{B,C\\ B^2C\ll T^{3+\alpha+\varepsilon}}}C^{1/2}\sum_{\substack{b,c}}w(\frac{b}{B})b^{-1}(\left(C/B\right)^{2/3}|m_1m_2^2|^{1/3}T^{-1-\varepsilon})^{-1/2}\\
    \ll&\ T^{2+2\delta+\varepsilon}m_1^{3+5/6} m_2^{2/3} \sum_{\substack{B,C\\ B^2C\ll T^{3+\alpha+\varepsilon}}}C^{1/6}B^{1/3}\\
    \ll&\ T^{5/2+2\delta+\alpha/6+5\eta+\varepsilon}\ .
\end{align*}
For small enough $\alpha,\delta$, and $\eta$, this is less than the claimed error term.\qed

\subsubsection{Bounding $\Sigma_{4}^{(2)}$}
The contribution of $\Sigma_{4}^{(2)}$ is
\begin{align*}
    &\sum_{\substack{b,c\\b^2c\ll T^{3-\alpha+\varepsilon}}}(b^2c)^{-1/2}\sum_{\epsilon=\pm}\sum_{\substack{D_2|D_1\\ c^2D_1=m_2D_2^2}}\frac{\tilde{S}(-\epsilon m_1,c^2,b^2,D_2,D_1)}{D_1D_2}\Phi_{w_4}^{W_{b,c}}\left(\frac{\epsilon b^2c^2m_1}{D_1D_2}\right)\ .
\end{align*}
As before, we can truncate the $D_2$ sum at $D_2\ll T^B$ for $B$ sufficiently large. We can also apply Lemma \ref{lemma:fine_bound_inttransfrom_voronoi} (c) and conclude that everything is negligible except if 
\[
    \frac{b^2c^2|m_1|}{D_1D_2}\geq T^{3-\varepsilon}\ .
\]
Let $D_1=mD_2$ for some $m\in\IN$. The above is equivalent to
\[
    b^2\frac{|m_2m_1|}{m^2}T^{-3+\varepsilon}\geq D_2.
\]
But for $\eta$ smaller than $\alpha/3$, this condition is not satisfied for nonzero $D_2$, and thus, everything is negligible for $\alpha$ large enough relative to $\eta$.\qed

\subsubsection{Bounding $\Sigma_{5}^{(1)}$}
The contribution of $\Sigma_{5}^{(1)}$ is
\begin{align*}
    &\sum_{\substack{b,c\\b^2c\ll T^{3+\alpha+\varepsilon}}}(b^2c)^{-1/2}\sum_{\epsilon=\pm}\sum_{\substack{D_1|D_2\\ m_1D_2=c^2D_1^2}}\frac{\tilde{S}(\epsilon b^2,m_1,m_2,D_1,D_2)}{D_1D_2}\Phi_{w_5}^{V_{b,c}}\left(\frac{\epsilon m_1m_2b^2}{D_1D_2}\right)\\
    =&\sum_{\substack{b,c\\b^2c\ll T^{3+\alpha+\varepsilon}}}(b^2c)^{-1/2}\sum_{\epsilon=\pm}\sum_{\substack{D_1\\ m_1\mid c^2D_1}}\frac{\tilde{S}(-\epsilon b^2,m_1,m_2,D_1,D_1^2c^2/m_1)}{D_1^3c^2/m_1}\Phi_{w_5}^{V_{b,c}}\left(\frac{\epsilon b^2m_2m_1^2}{D_1^3c^2}\right)\ .
\end{align*}
We again start by truncating the $D_1$ sum at $D_1\ll T^B$. We apply Lemma \ref{lemma:fine_bound_inttransfrom_voronoi} (a) to further truncate the sum at
\[
    D_1\leq \left(\frac{b}{c}\right)^{2/3}T^{-1+\varepsilon}|m_1^2m_2|^{1/3}\ .
\]
This seems very similar to $\Sigma_{4}^{(1)}$ but with $b$ and $c$ swapped. We observe that there are only possible values for $D_1$ if $b^2\gg T^{3-3\eta-\varepsilon}$ and thus $c\ll T^{\alpha+3\eta+\varepsilon}$. This will be used later. 

We argue in exactly the same way as for $\Sigma_{4}^{(1)}$. We start by summing over $b$ and $c$ in dyadic ranges $B,C$. We then perform Poisson summation in $b$ modulo $D_1$. We get the same result as before but without the factor $C^{1/2}$, as this came from $\frac{1}{c^{1/2}}$, and we now have $\frac{1}{b}$. We then use Lemma \ref{lemma:fine_bound_inttransfrom_voronoi} (c) for the non zero frequencies, which are negligible by the exact same argument with $b,c$ and $B,C$ swapped. We estimate the zero frequency trivially by $T^{\frac{3}{2}+2\delta+\varepsilon}$ using Lemma \ref{lemma:bound_fourier_voronoi}. Estimate all exponential sums trivially and bound the length of the $D_1$ sum by $T^{\frac{2}{3}\alpha+\eta+\varepsilon}$
and end up with
\[
    \sum_{\substack{B,C\\B^2C\ll T^{3+\alpha+\varepsilon}\\ C\ll T^{\alpha+3\eta+\varepsilon}}}\sum_{c}w(\frac{c}{C})c^{-1/2}T^{\frac{3}{2}+\varepsilon}T^{\frac{2}{3}\alpha+\eta+\varepsilon}\ll T^{\frac{3}{2}+2\delta+\frac{7}{6}\alpha+\frac{5}{2}\eta+\varepsilon}\ ,
\]
which is in the error term for small enough $\alpha,\delta$ and $\eta$.\qed

\subsubsection{Bounding $\Sigma_{5}^{(2)}$}
The contribution of $\Sigma_{5}^{(2)}$ is
\begin{align*}
    &\sum_{\substack{b,c\\b^2c\ll T^{3-\alpha+\varepsilon}}}(b^2c)^{-1/2}\sum_{\epsilon=\pm}\sum_{\substack{D_1|D_2\\ b^2D_2=m_1D_1^2}}\frac{\tilde{S}(\epsilon m_2,b^2,c^2,D_1,D_2)}{D_1D_2}\Phi_{w_5}^{W_{b,c}}\left(\frac{\epsilon b^2c^2m_2}{D_1D_2}\right)\ .
\end{align*}
We change variables $\tilde{D}_1:=\frac{m_1}{b^2}D_1$. This leads to
\begin{align*}
    &\sum_{\substack{b,c\\b^2c\ll T^{3-\alpha+\varepsilon}}}(b^2c)^{-1/2}\sum_{\epsilon=\pm}\sum_{\substack{\tilde{D}_1|D_2\\ m_1D_2=b^2\tilde{D}_1^2}}\frac{\tilde{S}(-\epsilon c^2,m_1,m_2,\tilde{D}_1,D_2)}{\tilde{D}_1D_2}\Phi_{w_5}^{W_{b,c}}\left(\frac{-\epsilon m_1m_2c^2}{\tilde{D}_1D_2}\right)\ .
\end{align*}
This is almost exactly the contribution of $\Sigma_{4}^{(1)}$. With the following differences: $m_1$ and $m_2$ are swapped, $\tilde{D}_1$ is $D_2$, and $D_2$ is $D_1$; there is an additional sign in the argument of $\Phi_{w_5}$; it is $\Phi_{w_5}$ instead of $\Phi_{w_4}$; the $b,c$ sum only goes to $b^2c\ll T^{3-\alpha+\varepsilon}$, and we have a different test function coming from the $W$-part of the approximate functional equation. But all the tools we used are indifferent to these changes. Lemma \ref{lemma:fine_bound_inttransfrom_voronoi} and Lemma \ref{lemma:bound_fourier_voronoi} also hold for $w_5$ and the other test function, and the sign is never used. The $\alpha$ part of the final bound changes, but $\alpha$ is so small that this does not matter anyway. Thus, this case also only contributes to the error term.\qed

\subsection{The long Weyl element}
\subsubsection{Bounding $\Sigma^{(1)}_6$}

The contribution of $\Sigma^{(1)}_6$ is
\[
    \sum_{b^2c\ll T^{3+\alpha+\varepsilon}} (b^2c)^{-1/2} \sum_{\epsilon_1,\epsilon_2 =\pm}\sum_{D_1,D_2} \frac{S(\epsilon_2c^2,m_2,m_1,\epsilon_1b^2;D_1,D_2)}{D_1D_2}\Phi_{w_6}^{V_{b,c}}\left(\frac{-\epsilon_2m_1c^2D_2}{D_1^2},\frac{-\epsilon_1m_2b^2D_1}{D_2^2}\right)\ .
\]
As before, we can truncate the $D_1$ and $D_2$ sums at $D_1,D_2\ll T^B$ for some large enough $B$.

We now replace the full Kloosterman sum with a sum over strata using Lemma \ref{lemma:b_f_KLsums} (c) and change the order of summation, resulting in
\begin{align*}
    \sum_{b^2c\ll T^{3+\alpha+\varepsilon}} (b^2c)^{-1/2} \sum_{\epsilon_1,\epsilon_2 =\pm}\sum_{D_0,D_1,D_2\ll T^{O(1)}} &\frac{S_{D_0}(\epsilon_2c^2,m_2,m_1,\epsilon_1b^2;D_0D_1,D_0D_2)}{D_0D_1D_0D_2}\\\cdot&\Phi_{w_6}^{V_{b,c}}\left(\frac{-\epsilon_2m_1c^2D_2}{D_1^2D_0},\frac{-\epsilon_1m_2b^2D_1}{D_2^2D_0}\right)\ .
\end{align*}

To truncate the sum further, we can use Lemma \ref{lemma:fine_bound_inttransfrom_long} (a). This implies that for
\[
    \min\left(\frac{c^{2/3}b^{1/3}m_1^{1/3}m_2^{1/6}}{(D_1D_0)^{1/2}},\frac{c^{1/3}b^{2/3}m_1^{1/6}m_2^{1/3}}{(D_2D_0)^{1/2}}\right)\leq T^{1-\varepsilon}
\]
we have 
\[
\Phi_{w_6}\left(\frac{-\epsilon_2m_1c^2D_2}{D_1^2D_0},\frac{-\epsilon_1m_2b^2D_1}{D_2^2D_0}\right) \ll T^{-B}
\]
for any positive $B$. Using $m_1,m_2\ll T^\eta$, we can truncate the sums at
\[
    D_1D_0\ll c^{4/3}b^{2/3}T^{-2+\eta+\varepsilon} \text{ and } D_2D_0\ll c^{2/3}b^{4/3}T^{-2+\eta+\varepsilon}\ .
\]
The $D_2,D_0$-sum is only nonempty if $b^2c \gg T^{3-\frac{3}{2}\eta-\varepsilon}$. Thus, we can truncate the $b,c$ sum further to
\[
    T^{3-\frac{3}{2}\eta-\varepsilon}\ll b^2c\ll T^{3+\alpha+\varepsilon}\ .
\]

We now sum over $b$ and $c$ in dyadic ranges $B$ and $C$ using a partition of unity given by a smooth, compactly supported function $w$ with bounded derivatives.

\begin{align*}
    &\sum_{T^{3-\frac{3}{2}\eta-\varepsilon}\ll B^2C\ll T^{3+\alpha+\varepsilon}} \sum_{\epsilon_1,\epsilon_2 =\pm}\sum_{\substack{D_0D_1\ll C^{4/3}B^{2/3}T^{-2+\eta+\varepsilon}\\D_0D_2\ll C^{2/3}B^{4/3}T^{-2+\eta+\varepsilon}}}\sum_{b,c}w(\frac{b}{B})w(\frac{c}{C}) (b^2c)^{-1/2} \\
    &\cdot\frac{S_{D_0}(\epsilon_2c^2,m_2,m_1,\epsilon_1b^2;D_1D_0,D_2D_0)}{D_1D_2D_0^2}\Phi_{w_6}^{V_{b,c}}\left(\frac{-\epsilon_2m_1c^2D_2}{D_1^2D_0},\frac{-\epsilon_1m_2b^2D_1}{D_2^2D_0}\right)\ .
\end{align*}

The explicit formula for the fine Kloosterman sum in Lemma \ref{lemma:b_f_KLsums} (c) implies that the exponential sum only depends on $c\mod D_1$, and we can perform Poisson summation in $c$ modulo $D_1$. We arrive at
\begin{align*}
    &\sum_{T^{3-\frac{3}{2}\eta-\varepsilon}\ll B^2C\ll T^{3+\alpha+\varepsilon}} \sum_{\epsilon_1,\epsilon_2 =\pm}\sum_{\substack{D_0D_1\ll C^{4/3}B^{2/3}T^{-2+\eta+\varepsilon}\\D_0D_2\ll C^{2/3}B^{4/3}T^{-2+\eta+\varepsilon}}}\sum_{b,c}b^{-1}w(\frac{b}{B})\\
    &\sum_{k_1\mathrm{mod} D_1}e\left(\frac{ck_1}{D_1}\right)\frac{S_{D_0}(\epsilon_2k_1^2,m_2,m_1,\epsilon_1b^2;D_1D_0,D_2D_0)}{(D_1D_0)^{2}D_2}\\
    &\cdot \int_\IR y_1^{-1/2}\Phi_{w_6}^{V_{b,y}}\left(\frac{-\epsilon_2m_1y_1^2D_2}{D_1^2D_0},\frac{-\epsilon_1m_2b^2D_1}{D_2^2D_0}\right)w(\frac{y_1}{C})e\left(\frac{y_1c}{D_1}\right)\mathrm{d}y_1\ .
\end{align*}
\subsubsection{The integral}
After a change of variables, we need to bound
\begin{align*}
    I:=C^{1/2}\int_\IR y_1^{-1/2}\Phi_{w_6}^{V_{b,y}}\left(\frac{-\epsilon_2m_1y_1^2C^2D_2}{D_1^2D_0},\frac{-\epsilon_1m_2b^2D_1}{D_2^2D_0}\right)w(y_1)e\left(\frac{y_1cC}{D_1}\right)\mathrm{d}y_1\ .
\end{align*}
We want to apply Lemma \ref{lemma:intbyparts} with $H(y_1)=\frac{cCy_1}{D_1}$ if $c$ is large and treat the cases $c=0$ and $c$ small separately. Assuming $c$ is nonzero and using Lemma \ref{lemma:fine_bound_inttransfrom_long} (b) to bound the derivatives of $\Phi_{w_6}^{V_{b,y}}$, we get
\[
    R=\frac{cC}{D_1} \textnormal{ and } U\gg T+\frac{m_1^{1/2}CD_2^{1/2}}{D_1D_0^{1/2}}+\frac{m_1^{1/3}m_2^{1/6}C^{2/3}B^{1/3}}{(D_1D_0)^{1/2}}\ .
\]
Thus, for 
\[
    c\gg T^{\varepsilon}\max(T\frac{D_1}{C},\frac{m_1^{1/2}D_2^{1/2}}{D_0^{1/2}},\frac{m_1^{1/3}m_2^{1/6}B^{1/3}}{C^{1/3}}\frac{D_1^{1/2}}{D_0^{1/2}})\ll T^{\frac{1}{3}\alpha+\frac{1}{2}\eta+\varepsilon}
\]
the integral is negligible.

For $c=0$, we have
\begin{align*}
    I=C^{1/2}\int_{(\varepsilon)}\big(\frac{b^2}{T^{3+\alpha}}\big)^{-u}\frac{e^{u^2}}{u}\int_\IR y_1^{-1/2-u}w(y_1)\Phi_{w_6}^{\tilde{V}_u}\left(\frac{-\epsilon_2y_1^2m_1C^2D_2}{D_1^2D_0},\frac{-\epsilon_1m_2b^2D_1}{D_2^2D_0}\right)\mathrm{d}y_1\frac{\mathrm{d}u}{2\pi i}\ .
\end{align*}
We have
\[
    \Big|\frac{-\epsilon_1m_2b^2D_1}{D_2^2D_0}\Big|\ll B^{8/3}C^{4/3}T^{-2+2\eta+\varepsilon}\ll T^{2+\frac{4}{3}\alpha+2\eta+\varepsilon}\ .
\]
By the rapid decay of $e^{u^2}$ in vertical lines, we can truncate the $u$ integral at $|u|\leq T^\varepsilon$. By Lemma \ref{lemma:testfunctions} (a), $\tilde{V}_u\in \CH(A_0,\tilde{\mu})$, and for $\alpha$ and $\eta$ small enough, we can apply Lemma \ref{lemma:bound_fourier_long} (a) with $s=\frac{1}{2}+u$ and the test function $\tilde{V}_u$ to the $y_1$ integral. This leads to the total bound for the integral in this case of $C^{1/2}T^{\frac{3}{2}+2\delta-\frac{1}{6}+\frac{1}{360}+\varepsilon}$.

For $0\neq |c|\ll T^{\frac{1}{3}\alpha+\frac{1}{2}\eta+\varepsilon}$, we also have 
\begin{align*}
    \Big|\frac{-\epsilon_1m_2b^2D_1}{D_2^2D_0}\Big|\ll B^{8/3}C^{4/3}T^{-2+2\eta+\varepsilon}\ll T^{2+\frac{4}{3}\alpha+2\eta+\varepsilon}\ .
\end{align*}
In addition, we have
\begin{align*}
    \Big|\frac{cC}{D_1}\Big|\gg T^{1-\frac{1}{3}\alpha-\eta+\epsilon},\ 
    \left|\left(\frac{-\epsilon_1m_2b^2D_1}{D_2^2D_0}\right)\left(\frac{-\epsilon_2m_1C^2D_2}{D_1^2D_0}\right)^{-1}\left(\frac{cC}{D_1}\right)^3\right|=\left|\frac{m_2b^2Cc^3}{m_1D_2^3}\right|
    \\\textnormal{ and }T^{3-2\alpha-4\eta-\epsilon}\ll \left|\frac{m_2b^2Cc^3}{m_1D_2^3}\right|\ll T^{3+2\alpha+\eta+\varepsilon}\ .
\end{align*}
As above, we truncate the $u$ integral to $|u|\leq T^\varepsilon$ at the cost of a negligible error. For $\alpha$ and $\eta$ small enough, we apply Lemma \ref{lemma:bound_fourier_long} (b) with $s=\frac{1}{2}+u$ and the test function $\tilde{V}_u$, and then multiply by the length of the $c$ sum to get the bound
\[
    C^{1/2}T^{\frac{3}{2}+\frac{1}{3}\alpha+\frac{1}{2}\eta -\frac{1}{240}+2\delta+\varepsilon}
\]
for the contribution of the integral and the $c$ sum in this case.

\subsubsection{The exponential sum}
We want to bound 
\[
\sum_{k_1\mathrm{mod} D_1}e\left(\frac{ck_1}{D_1}\right)\frac{S_{D_0}(\epsilon_2k_1^2,m_2,m_1,\epsilon_1b^2;D_1D_0,D_2D_0)}{(D_1D_0)^{2}D_2}\ .
\] Ignoring the factor $(D_1D_0)^{-2}D_2^{-1}$ for now, we start by inserting the parametrization of the fine Kloosterman sum from Lemma \ref{lemma:b_f_KLsums} (c). This gives
\begin{align*}
    &\sum_{k_1\mod D_1}e\left(\frac{ck_1}{D_1}\right)\sum_{\substack{x\mod D_0\\ xy\equiv 1\mod D_0\\ m_1D_2+m_2D_1x\equiv 0\mod D_0} }S(\epsilon_2k_1^2,\frac{m_1D_2+m_2D_1x}{D_0};D_1)S(\epsilon_1b^2,\frac{m_1D_2y+m_2D_1}{D_0};D_2)\\
    =&\sum_{\substack{x\mod D_0\\ xy\equiv 1\mod D_0\\ m_1D_2+m_2D_1x\equiv 0\mod D_0}}S(\epsilon_1b^2,\frac{m_1D_2y+m_2D_1}{D_0};D_2)\sum_{k_1\mod D_1}e\left(\frac{ck_1}{D_1}\right)S(\epsilon_2k_1^2,\frac{m_1D_2+m_2D_1x}{D_0};D_1)\ .
\end{align*}
Let $N(x)=\frac{m_1D_2+m_2D_1x}{D_0}$. Looking at the $k_1$ sum and the second Kloosterman sum, we have
\begin{align*}
    \sum_{k_1\mod D_1}e\left(\frac{ck_1}{D_1}\right)S(\epsilon_2k_1^2,\frac{m_1D_2+m_2D_1x}{D_0};D_1)
    =\sum_{k_1\mod D_1}e\left(\frac{ck_1}{D_1}\right)S(k_1^2,\epsilon_2N(x);D_1)\ .
\end{align*}
Using Lemma \ref{lemma:exp_sum_bound}, this can be bounded independently of $x$ and $c$ by $s_{D_1}D_1^{1+\varepsilon}$ with $D_1=s_{D_1}^2f_{D_1}$ and $f_{D_1}$ squarefree. We bound the Kloosterman sum modulo $D_2$ trivially by $D_2$ and bound the length of the $x$ sum by $D_0$. Including the factor $(D_1D_0)^{-2}D_2^{-1}$, we end up with the bound $s_{D_1}D_1^{-1+\varepsilon}D_0^{-1}$.

\subsubsection{Bringing it all together}
Using the above bounds for the integral and the exponential sum, we are left with 
\begin{align*}
    &\sum_{T^{3-\frac{3}{2}\eta-\varepsilon}\ll B^2C\ll T^{3+\alpha+\varepsilon}} \sum_{\epsilon_1,\epsilon_2 =\pm}\sum_{\substack{D_0D_1\ll C^{4/3}B^{2/3}T^{-2+\eta+\varepsilon}\\D_0D_2\ll C^{2/3}B^{4/3}T^{-2+\eta+\varepsilon}}}\sum_{b}b^{-1}w(\frac{b}{B}) \\
    &\cdot s_{D_1}D_1^{-1+\varepsilon}D_0^{-1}C^{1/2}T^{\frac{3}{2}+\frac{2}{3}\alpha -\frac{1}{240}+\varepsilon}\ .
\end{align*}
The $D_0$ sum can be bounded by $T^{\varepsilon}$, and the $D_2$ sums can be bounded by $T^{\frac{2}{3}\alpha+\eta+\varepsilon}$. Let $X=C^{4/3}B^{2/3}T^{-2+\eta+\varepsilon}$. For the $D_1$ sum, we have the following:
\begin{align*}
    \sum_{D_1\ll X}s_{D_1}D_1^{-1+\varepsilon}=\sum_{s_{D_1}\ll X^{1/2}}\sum_{f\ll X/s_{D_1}^2}(fs_{D_1})^{-1+\varepsilon}\mu(f)^2\ll X^\varepsilon\ll T^\varepsilon.
\end{align*}
The $b$ sum is also only of size $T^\varepsilon$. The sum over $\epsilon_1,\epsilon_2$ is of length $4$. The biggest value for $C^{1/2}$ is $T^{\frac{3}{2}+\frac{1}{2}\alpha+\varepsilon}$. Putting it all together, we are left with the bound
\[
    T^{3+2\delta-\frac{1}{240}+\frac{11}{6}\alpha+\eta+\varepsilon}\ ,
\]
which is enough for very small $\alpha,\delta$ and $\eta$.\qed

\subsubsection{Bounding $\Sigma^{(2)}_6$}
The contribution of $\Sigma^{(2)}_6$ is
\[
    \sum_{b^2c\ll T^{3-\alpha+\varepsilon}} (b^2c)^{-1/2} \sum_{\epsilon_1,\epsilon_2 =\pm}\sum_{D_1,D_2} \frac{S(\epsilon_2m_1,c^2,b^2,\epsilon_1m_2;D_1,D_2)}{D_1D_2}\Phi_{w_6}^{W_{b,c}}\left(\frac{-\epsilon_2m_1b^2D_2}{D_1^2},\frac{-\epsilon_1m_2c^2D_1}{D_2^2}\right)\ .
\]
We again truncate the $D_1,D_2$ sums at $D_1,D_2\leq T^B$ for some sufficiently large $B$.

We then use Lemma \ref{lemma:fine_bound_inttransfrom_long} (c). This implies that for
\[
    \min\left(\frac{b^{2/3}c^{1/3}m_1^{1/3}m_2^{1/6}}{D_1^{1/2}},\frac{c^{2/3}b^{1/3}m_2^{1/3}m_1^{1/6}}{D_2^{1/2}}\right)\leq T^{1-\varepsilon}
\]
we have 
\[
\Phi_{w_6}^{W_{b,c}}\left(\frac{-\epsilon_2m_1b^2D_2}{D_1^2},\frac{-\epsilon_1m_2c^2D_1}{D_2^2}\right) \ll_B T^{-B}
\]
for any positive $B$. We can thus truncate the sums at
\[
    D_1\ll c^{2/3}b^{4/3}T^{-2+\eta+\varepsilon} \text{ and } D_2\ll c^{4/3}b^{2/3}T^{-2+\eta+\varepsilon}\ .
\]
The $D_1$-sum is only nonempty if $b^2c \gg T^{3-\frac{3}{2}\eta-\varepsilon}$, which is not the case for $\alpha>\frac{3}{2}\eta+\varepsilon$. Thus, the whole contribution of $\Sigma^{(2)}_6$ is negligible for $\alpha$ big enough compared to $\eta$.\qed

\end{document}